\documentclass{amsart}

\swapnumbers \theoremstyle{plain}
\newtheorem{thm}{Theorem}[section]
\newtheorem{lem}[thm]{Lemma}

\newtheorem{lem-defn}[thm]{Lemma and definition}

\newtheorem{cor}[thm]{Corollary}
\newtheorem{prop}[thm]{Proposition}
\theoremstyle{definition}\newtheorem{rem}[thm]{Remark}
\newtheorem{defn}[thm]{Definition}

\theoremstyle{definition}\newtheorem{leer}[thm]{}

\newcommand{\Cal}{\mathcal}

\newcommand{\R}{\mathbb{R}}

\newcommand{\C}{\mathbb{C}}
\newcommand{\D}{\mathbb{D}}
\newcommand{\N}{\mathbb{N}}
\newcommand{\Z}{\mathbb{Z}}
\newcommand{\E}{\mathbb{E}}

\newcommand{\B}{\mathbb{B}}
\newcommand{\s}{\mathbb{S}}

\DeclareMathOperator{\supp}{supp} \DeclareMathOperator{\im}{Im}
\DeclareMathOperator{\ke}{Ker} \DeclareMathOperator{\rea}{Re}

 \numberwithin{equation}{section}

\usepackage{soul,cancel,url}
\usepackage{amsfonts}
\usepackage{hyperref}
\usepackage{color}
\usepackage{amsmath}
    \usepackage{amsxtra}
    \usepackage{amstext}
    \usepackage{amssymb}
   \usepackage{latexsym}
   \usepackage{graphicx}

\title {Local and global similarity of holomorphic matrices}
\author{J\"urgen Leiterer}
\address{Institut f\"ur Mathematik \\
Humboldt-Universit\"at zu Berlin \\Rudower Chaussee 25\\D-12489 Berlin , Germany}
\email{leiterer@mathematik.hu-berlin.de}

\date{}

\begin{document}


\maketitle

\section{Introduction}

Let $X$ be a  (reduced) complex space, e.g., a complex manifold or an analytic subset of a complex manifold. Denote by
 $\mathrm{Mat}(n\times n,\C)$ the algebra of complex $n\times n$ matrices, and by $\mathrm{GL}(n,\C)$  the group of all invertible complex $n\times n$ matrices.

\begin{defn}\label{7.6.16-} Two holomorphic maps $A,B:X\to \mathrm{Mat}(n\times n,\C)$ are called
(globally) {\bf  holomorphically  similar on $ X$} if there is a holomorphic  map $H:X\to\mathrm{GL}(n,\C)$ with $B=H^{-1}AH$ on $X$.
They are called
{\bf locally holomorphically similar at a point $\xi\in X$}  if there is a neighborhood $U$ of $\xi$ such that $A\vert_U$ and $B\vert_U$ are holomorphically similar on $U$.
Correspondingly we define {\bf  continuous} and {\bf $\Cal C^k$ similarity}.
\end{defn}

\begin{thm}\label{28.6.16*}
Let $X$ be a  one-dimensional Stein space, and let $A,B:X\to \mathrm{Mat}(n\times n,\C)$ be two holomorphic maps, which  are locally holomorphically similar at each point of $X$. Then they are globally holomorphically similar on $X$.
\end{thm}

If $X$ is a non-compact connected Riemann surface, this was proved by
R. Guralnick \cite{Gu}. Actually, Guralnick proves a more general theorem for matrices with elements in some Bezout rings, and then  applies this to the ring  of holomorphic functions on a non-compact connected Riemann surface. The ring of holomorphic functions on an arbitrary (non-smooth) one-dimensional Stein space is not  Bezout. So it seems  that this proof cannot be directly generalized to the non-smooth case.

We give a new proof which works also in the non-smooth cases. Moreover, this approach  applies  to higher dimensions, where we get the following.

\begin{thm}\label{24.8.16+}Let $X$ be a Stein space and let $A,B:X\to \mathrm{Mat}(n\times n,\C)$ be two  holomorphic maps, which are globally $\Cal C^\infty$ similar on $X$.
Then they are globally holomorphically similar on $X$.
\end{thm}

We show by an example that $\Cal C^\infty$ cannot be replaced by $\Cal C^k$ with $k<\infty$ (Theorem \ref{10.1.16'}), even if we additional assume that $A$ and $B$ are locally holomorhically similar at each point of $X$ -- under the hypotheses of Theorem \ref{24.8.16+}, this follows from Spallek's criterion  (Theorem \ref{28.6.16**}, condition (iii)).

There are different criteria for local holomorphic similarity, which are known or can be easily obtained from (partially, non-easy) known results. They are contained  in the following theorem (if the there given matrix $\Phi$ is invertible).

\begin{thm}\label{28.6.16**} Let $X$ be a complex space,  $A,B:X\to \mathrm{Mat}(n\times n,\C)$  holomorphic, $\xi\in X$ and $\Phi\in \mathrm{Mat}(n\times n,\C)$ such that $\Phi B(\xi)=A(\xi)\Phi$.
Suppose at least one of the following conditions is satisfied.
\begin{itemize}
\item[(i)] {\em (Wasow's criterion)} The dimension of the complex vector space
\begin{equation}\label{30.6.16}
\Big\{\Theta\in \mathrm{Mat}(n\times n,\C)\;\Big\vert\;\Theta B(\zeta)=A(\zeta)\Theta\Big\}
\end{equation}is constant for $\zeta$ in some neighborhood of $\xi$.
\item[(ii)]  {\em(Smith's criterion)} $X$ is a Riemann surface,  and there exist a neighborhood $V$ of $\xi$ and a continuous map $C:X\to \mathrm{GL}(n,\C)$ with
$C(\xi)=\Phi$ and $CB=AC$ on $V$.
\item[(iii)] {\em (Spallek's criterion)} There exist a neighborhood $V$ of $\xi$ and a $\Cal C^\infty$ map $T:V\to \mathrm{Mat}(n\times n,\C)$ such that
$T(\xi)=\Phi$ and $TB=AT$ on $V$.
\end{itemize}
Then there exist a neighborhood $U$ of $\xi$ and a holomorphic map $H:U\to \mathrm{Mat}(n\times n,\C)$ such that $H(\xi)=\Phi$ and $HB=AH$ on $U$.
\end{thm}

Proofs or  references for the statements contained in this theorem will be given in  Section \ref{local}.
Of course, condition (i) and (ii) are not necessary, whereas condition (iii) is.
We show by examples (Theorem \ref{6.1.16''}) that, in condition (iii), $\Cal C^\infty$ cannot be replaced by $\Cal C^k$ with $k<\infty$ (the same $k$ for all $A$, $B$ and $\xi$). However, see Remark \ref{3.9.16*}.

Theorems \ref{28.6.16*} and \ref{28.6.16**} together give the following

\begin{thm}\label{28.6.16***}
Let $X$ be a  one-dimensional Stein space and $A,B:X\to \mathrm{Mat}(n\times n,\C)$  holomorphic. Then, for the global holomorphic similarity of $A$ and $B$ it is sufficient that, for each point $\xi\in X$, at least one of the following holds.

\begin{itemize}
\item[(i)] There exists $\Phi\in \mathrm{GL}(n,\C)$ with $B(\xi)=\Phi^{-1}A(\xi)\Phi$, and the dimension of the vector space \eqref{30.6.16} is constant for all  $\zeta$ in some neighborhood of $\xi$.

\item[(ii)] $\xi$ is a smooth point of $X$, and $A$ and $B$ are locally $\Cal C^0$ similar at $\xi$.

\item[(iii)] $A$ and $B$ are locally $\Cal C^\infty$ similar at $\xi$.

\end{itemize}
\end{thm}

In particular, this contains the following slight improvement of Guralnick's result.

\begin{thm}\label{19.6.16}Let $X$ be a non-compact connected Riemann surface, and  $A,B:X\to \mathrm{Mat}(n\times n,\C)$  holomorphic. If $A$ and $B$  are locally continuously similar at each point of $X$, then they   are globally holomorphically similar on $X$.
\end{thm}

{\bf Acknowledgements:} I want to thank F. Forstneri\u c and J. Ruppenthal for helpful discussions (in particular, see Remark \ref{3.9.16**}).

\section{Notations}\label{20.6.16}

$\N$ is the set of natural numbers including $0$. $\N^*=\N\setminus\{0\}$.

If $n,m\in \N^*$, then by $\mathrm{Mat}(n\times m,\C)$ we denote the space of complex $n\times m$ matrices ($n$ rows and $m$ columns), and by
$\mathrm{GL}(n,\C)$ the  group of invertible complex $n\times n$ matrices.

The unit matrix in $\mathrm{Mat}(n\times n,\C)$ will be denoted by $I_n$ or simply by $I$.

$\ke \Phi$ denotes the kernel,  $\im \Phi$ the image and $\Vert\Phi\Vert$ the operator norm of a matrix $\Phi\in \mathrm{Mat}(n\times m,\C)$ considered as a linear map from $\C^m$ to $\C^n$.

By a complex space we always mean a {\em reduced} complex space or, using, e.g., the terminology of \cite{L}, an {\em analytic space}..

{\bf On the use of the language of sheaves in this paper:}

Let $X$ be a  topological  space,  and  $G$ a topological group (abelian or non-abelian).
Then we denote by $\Cal C^G_X$, or simply by $\Cal C^G$, the sheaf of continuous $G$-valued maps on $X$, i.e.,  for each non-empty open $U\subseteq X$,
 $\Cal C^{G}_X(U)= \Cal C^{G}(U)$ is the group of continuous maps $f:U\to {G}$, and  $\Cal C^G_X(\emptyset)=\Cal C^G(\emptyset)$ is the neutral element of $G$.

  If $X$ is  a complex space and  $G$ is a complex Lie group, then we denote by $\Cal O^{G}_X$, or simply by $\Cal O^{G}$,  the  subsheaf of holomorphic maps of $\Cal C^G_X$.

All sheaves in this paper are subsheaves of $\Cal C_X^G$ (for some $X$ and some $G$).

Let $\Cal F$ be such a sheaf.

Let $\Cal U=\{U_i\}_{i\in I}$  an open covering of $X$.

A family $f_{ij}\in\Cal F(U_i\cap U_j)$, $i,j\in I$, is called a {\bf $(\Cal U,\Cal F)$-cocycle} if (with the group operation in $G$ written as a multiplication)
\begin{equation*}
f_{ij}f_{jk}=f_{ik}\quad\text{on}\quad U_i\cap U_j\cap U_k\quad\text{for all}\quad i,j,k\in I.\quad\footnote{Here and in the following we use the  convention that  statements like
``$f=g$ on $\emptyset$'' or ``$f:=g$ on $\emptyset$''
have to be omitted.}
\end{equation*}Note that then always $f^{-1}_{ij}=f^{}_{ji}$ and $f^{}_{ii}$ is identically equal to the neutral element of $G$. The set of all $(\Cal U,\Cal F)$-cocycles will be denoted by $Z^1(\Cal U,\Cal F)$.
Two cocycles  $\{f_{ij}\}$ and $\{g_{ij}\}$ in $Z^1(\Cal U,\Cal F)$ are called {\bf $\Cal F$-equivalent} if there exists a family $h_i\in \Cal F(U_i)$, $i\in I$, such that
\[
f^{}_{ij}=h^{}_ig^{}_{ij}h^{-1}_j\quad\text{on}\quad U_i\cap U_j\quad\text{for all}\quad i,j\in I.
\] If, in this case, for all $i,j$, the map $g_{ij}$ is identically equal to the neutral element of $G$,  then $f$ is called {\bf $\Cal F$-trivial}.

We say that $f$ is an {\bf $\Cal F$-cocycle} (on $X$), if there exists an open covering $\Cal U$ of $X$ with $f\in Z^1(\Cal U,\Cal F)$. This covering then is called {\bf the covering of $f$}.
As usual we write
\[H^1(X,\Cal F)=0
\] to say that each $\Cal F$-cocycle is $\Cal F$-trivial.

Let $\Cal U=\{U_i\}_{i\in I}$ and $\Cal U^*=\{V_\alpha\}_{\alpha\in I^*}$  be two  open coverings of $X$ such that $\Cal U^*$ is a refinement of $\Cal U$, i.e., there is a map
$\tau:I^*\to I$ with $U^*_\alpha\subseteq U_{\tau(\alpha)}$ for all $\alpha\in I^*$.
Then we say  that
a $(\Cal U^*,\Cal F)$-cocycle $\{f^*_{\alpha}\}_{\alpha,\beta\in I^*}$  is {\bf induced} by a $(\Cal U,\Cal F)$-cocycle $\{f_{ij}\}_{i,j\in I}$
if this map $\tau$ can be chosen so that
\[
f^*_{\alpha\beta}=f_{\tau(\alpha)\tau(\beta)}^{}\quad\text{on}\quad U^*_i\cap U^*_j\quad\text{for all}\quad \alpha,\beta\in I^*.
\]
We will frequently use the following simple and well-known proposition, see \cite[p. 101]{C} for ``if'' and  \cite[p. 41]{Hi} for ``only if''.
\begin{prop}\label{17.12.15--} Let $f,g\in Z^1(\Cal U,\Cal F)$ and $f^*,g^*\in Z^1(\Cal U^*,\Cal F)$ such that $f^*$ and $g^*$ are  induced by $f$ and $g$, respectively.
Then $f$ and $g$ are $\Cal F$-equivalent if and only if $f^*$ and $g^*$ are $\Cal F$-equivalent.
\end{prop}

Now let $\Cal U$ and $\Cal V$ be two arbitrary open coverings of $X$, $f\in Z^1(\Cal U,\Cal F)$ and $g\in Z^1(\Cal V,\Cal F)$.
Then we say that $f$ and $g$ are {\bf $\Cal F$-equivalent} if there exist an open covering $\Cal W$ of $X$, which is  a refinement of both $\Cal U$ and
$\Cal V$, and $(\Cal W,\Cal F)$ cocycles $f^*$ and $g^*$, which are induced by $f$ and $g$, respectively, such that $f^*$ and $g^*$ are $\Cal F$
equivalent. By Proposition \ref{17.12.15--}, this definition is in accordance with the definition of equivalence given above for $\Cal U=\Cal V$.

Let $Y$ be a non-empty open subset of $X$.

Then we denote by $\Cal F\vert_Y$ the sheaf  defined by $\Cal F\vert_Y(U)=\Cal F(U)$ for each open $U\subseteq Y$. $\Cal F\vert_Y$ is called the {\bf restriction} of $\Cal F$ to $Y$. If $f$ is an $\Cal F$-cocycle and $\Cal U=\{U_i\}_{i\in I}$ is the covering of $f$, then we denote by $f\vert_Y=\{(f\vert_Y)^{}_{ij}\}_{i,j\in I}$ the $\Cal F\vert_Y$-cocycle  with the covering $\Cal U\cap Y:=\big\{U_i\cap Y\big\}_{i\in I}$  defined by
\[
(f\vert_Y)^{}_{ij}=f^{}_{ij}\big\vert_{U_i\cap U_j\cap Y}\quad\text{for}\quad i,j\in I.
\] We call $f\vert_Y$ the {\bf restriction} of $f$ to $Y$. We say that $f\vert_Y$ is $\Cal F$-trivial if $f\vert_Y$ is $\Cal F\vert_Y$-trivial.

\section{Proof of Theorem \ref{28.6.16**}{}}\label{local}

We show that the statements of this theorem are known or easily  follow from known results.
First we collect these known results.

We begin with following deep result of K. Spallek, which is a special case of  \cite[Satz 5.4]{Sp1} (see also the beginning of \cite{Sp2}).

\begin{lem}\label{8.12.15'}Let $X$ be a complex space, $M:X\to \mathrm{Mat}(n\times m,\C)$  holomorphic, and $\xi\in X$. Then there exists
$k\in \N$ (depending on $M$ and $\xi$) such that the following holds.

Suppose $U$ is a neighborhood of $\xi$ and $f:U\to \C^m$ is  a $\Cal C^k$ map with
$Mf=0$ on $U$.
Then there exist a neighborhood $V\subset U$ of $\xi$ and a holomorpic map  $h:V\to \C^m$ such that $Mh=0$ on $V$ and $h(\xi)=f(\xi)$.
\end{lem}

The following lemma is more easy to prove and nowadays well-known. Proofs can be found, e.g., in
\cite[Corollary 2]{Sh} or  \cite{W}.\footnote{Lemma \ref{9.12.15} is not explicitly stated in \cite{W},
 but it follows immediately from Lemma 1 of \cite{W}. Also, in \cite{W}, $X$ is a domain in the complex plane,  but the proof given there works
also in the general case.}
\begin{lem}\label{9.12.15} Let $X$ be a complex space, $M:X\to \mathrm{Mat}(n\times m,\C)$ holomorphic, and $\xi\in X$ such that the dimension of $\ke M(\zeta)$ does not depend on $\zeta$ in some neigborhood of $\xi$.   Then
 there exist a neighborhood $U$  of $\xi$ and a holomorphic map  $P:U\to \mathrm{Mat}(m\times m,\C)$ such that
\begin{align}\label{25.8.16}&P(\zeta)^2=P(\zeta)\quad\text{for all}\quad\zeta\in U,\\
&\label{9.12.15''''}\ke M(\zeta)=\im P(\zeta)\quad\text{for all}\quad\zeta\in U.
\end{align}In other words, $\{\ke M(\zeta)\}_{\zeta\in U}$ is a holomorphic subvector bundle of $U\times \C^m$.
\end{lem}
\begin{cor}\label{11.12.15-}Under the hypothesis of Lemma \ref{9.12.15}, for each vector $v\in \C^n$ with $M(\xi)v=0$, there
exist a neighborhood $U$ of $\xi$ and a holomorphic map  $h:U\to E$ such that $Mh=0$ on $V$ and $h(\xi)=v$.
\end{cor}
\begin{proof} Set $h(\zeta)=P(\zeta)v$, where $P$ is as in Lemma \ref{9.12.15}.
\end{proof}

Also the following lemma seems to be well-known. To my knowledge, for the first time it was observed  in \cite[p. 200]{BKL}. For completeness we supply the proof mentioned there.

\begin{lem}\label{23.6.16''} Let $X$ be a Riemann surface, and let $M:X\to \mathrm{Mat}(n\times m,\C)$ be holomorphic. Then, for each $\xi\in X$, there are a neighborhood $U$ and a holomorphic map $P:U\to \mathrm{Mat}(m\times m,\C)$ such that
\begin{align*}&P(\zeta)^2=P(\zeta)\quad\text{for all}\quad \zeta\in U,\\
& \im P(\zeta)= \ke M(\zeta)\quad\text{for all}\quad\zeta\in U\setminus\{\xi\},\\
&\im P(\xi)\subseteq \ke M(\xi).
\end{align*}In other words, there exists a holomorphic subvector bundle $\{K(\zeta)\}_{\zeta\in U}$ of $U\times \C^m$ such that $K(\zeta)=\ke M(\zeta)$ for all $\zeta\in U\setminus\{\xi\}$ and $K(\xi)\subseteq \ke M(\xi)$.
\end{lem}
\begin{proof} Set $r=\max_{\zeta\in X}\dim \im M(\zeta)$.
If $r= 0$, then $\varphi\equiv 0$ and  $P(\zeta):=I$, $\zeta\in X$, has the required properties. Let $r>0$.
Then (we may assume that $X$ is connected), by the Smith factorization theorem  (see, e.g., \cite[Ch. III, Sect. 8]{J},
applied to the ring of germs of holomorphic functions in neighborhoods of $\xi$ (a direct proof for that ring can be found, e.g., in \cite[\S 1.3]{GS} and \cite[Theorem 4.3.1]{GL}), we can find an open neighborhood $U$  of $\xi$, holomorphic maps $E:U\to \mathrm{GL}(n,\C)$, $F:U\to \mathrm{GL}(m,\C)$, and nonnegative  integers $\kappa_{1},\ldots,\kappa_{r}$
such that
\begin{equation*}
M(\zeta)=E(\zeta)\begin{pmatrix}\Delta(\zeta)&0\\0&0\end{pmatrix}F(\zeta)\quad\text{for all}\quad \zeta\in U,
\end{equation*}where $\Delta(\zeta)$ is the diagonal matrix with diagonal $(\zeta-\xi)^{\kappa_{1}},\ldots,(\zeta-ßxi)^{\kappa_{r}}$.
Then
\[
P(\zeta):=F^{-1}(\zeta)\begin{pmatrix}0&0\\0&I_{m-r}\end{pmatrix}F(\zeta),\quad\zeta\in U,
\]has the required properties.\footnote{If $r=\min(n,m)$ some of the zeros in the block matrices have to be omitted.}
\end{proof}

{\em Proof of Theorem \ref{28.6.16**}.}   Let $\mathrm{End} (\mathrm{Mat}(n\times n,\C))$ be the space of linear endomorphisms of $\mathrm{Mat}(n\times n,\C)$, and let $\varphi^{}_{A,B}:X\to \mathrm{End} (\mathrm{Mat}(n\times n,\C))$  be  the holomorphic map
defined by
\[
\varphi_{A,B}(\zeta)\Phi= A(\zeta)\Phi-\Phi B(\zeta)\quad\text{for}\quad \zeta\in X\text{ and }\Phi\in \mathrm{Mat}(n\times n,\C).
\]Fix a basis of $\mathrm{Mat}(n\times n,\C)$, and let $M_{A,B}^{}$ be the representation matrix of $\varphi^{}_{A,B}$ with respect to this basis.

{\em Proof under hypothesis} (i). This was first proved by W. Wasow \cite{W}. Wasow  considered only the case when $X$ is a domain in $\C$, but his proof works also in the general case. It goes as follows:

By definition of $\varphi^{}_{A,B}$, \eqref{30.6.16} is the kernel of $\varphi^{}_{A,B}(\zeta)$. Therefore, we have
a neighborhood $U$ of $\xi$ and a number $r\in\N$ such that
\[
\dim\ke\varphi(\zeta)=r\quad\text{for all}\quad \zeta\in U.
\] By Lemma \ref{9.12.15}, this means that the family $\{\ke \varphi^{}_{A,B}(\zeta)\}_{\zeta\in U}$ is a holomorphic subvector bundle of the product bundle $U\times \mathrm{Mat}(n\times n,\C)$. Since $\Phi\in \ke \varphi^{}_{A,B}(\xi)$, then,  after shrinking $U$, we we can find a holomorphic section $H$ of this bundle with $H(\xi)=\Phi$. \qed

{\em Proof under hypothethis} (ii).
From Lemma \ref{23.6.16''} applied  to $M^{}_{A,B}$, we get a  neighborhood $U$ of $\xi$ and a holomorphic map $P:U\to \mathrm{End} (\mathrm{Mat}(n\times n,\C))$ satisfying
\begin{align}&\label{24.6.16neu}P(\zeta)^2=P(\zeta)\quad\text{for all}\quad \zeta\in U,\\
&\label{24.6.16'neu} \im P(\zeta)= \ke \varphi^{}_{A,B}(\zeta)\quad\text{for all}\quad\zeta\in U\setminus\{\xi\},\\
&\label{24.6.16''neu}\im P(\xi)\subseteq \ke \varphi^{}_{A,B}(\xi).
\end{align}
Moreover, since $CB=AC$ on $U$, then
$C(\zeta)\in \ke \varphi^{}_{A,B}(\zeta)$ for all $\zeta\in U$. By \eqref{24.6.16'neu}, this implies that $C(\zeta)\in \im P(\zeta)$ for all $\zeta\in U\setminus\{\xi\}$. By \eqref{24.6.16neu} this further implies that
$P(\zeta)\big(C(\zeta)\big)=C(\zeta)$ for all $\zeta\in U\setminus\{\xi\}$ and, by continuity,
\[
P(\xi)\big(\Phi\big)=P(\xi)\big(C(\xi)\big)=C(\xi)=\Phi.
\] Define a holomorphic map $H:U\to\mathrm{Mat}(n\times n,\C)$ by
\[
H(\zeta)=P(\zeta)\big(\Phi\big),\quad\zeta\in U.
\] Then $H(\xi)=\Phi$ and, by \eqref{24.6.16'neu} and \eqref{24.6.16''neu}, $HB=AH$ on $U$.  \qed

{\em Proof under hypothesis} (iii). Then from Spallek's theorem (Lemma \ref{8.12.15'} with $M$ a representation matrix of $\varphi^{}_{A,B}$ and $v=\Phi$) we get a
neighborhood $U$ of $\xi$ and a holomorphic map  $H:U\to \mathrm{Mat}(n\times n,\C)$ such that $H(\xi)=\Phi$ and $H(\zeta)\in \ke\varphi^{}_{A,B}(\zeta)$ for all $\zeta\in U$, i.e., $HB=AH$ on $U$. \qed

\begin{rem}\label{3.9.16*} The proof under hypothesis (iii) shows that actually the following holds:
There exists a  positive integer $k$ depending on $\xi$ and $A$ such that, if there exist a neighborhood $U$ of $\xi$ and a $\Cal C^k$ map $T:V\to \mathrm{Mat}(n\times n,\C)$ such that $T(\xi)=\Phi$ and $TB=AT$ on $U$,
then  there exist a neighborhood $V\subseteq U$ of $\xi$ and a holomorphic map $H:V\to \mathrm{Mat}(n\times n,\C)$ such that  $H(\xi)=\Phi$ and
$HB=AH$ on $V$.
\end{rem}

\section{An Oka principle and proof of Theorem \ref{24.8.16+}}

\begin{defn}\label{6.6.16''} Let  $\Phi\in\mathrm{Mat}(n\times n,\C)$.
We denote by $\mathrm{Com\,} \Phi$ the algebra of all $\Theta\in \Phi\in\mathrm{Mat}(n\times n,\C)$ with $\Phi \Theta=\Theta \Phi$, and by $\mathrm{GCom\,} \Phi$ we denote the group of invertible elements of $\mathrm{Com\,} \Phi$.
Note that, as easily seen,
\begin{align}&\label{17.8.16+}\mathrm{GCom\,} \Phi=\mathrm{GL}(n,\C)\cap \mathrm{Com\,} \Phi,\text{ and}\\
&\label{26.8.16+}
 \mathrm{Com\,}(\Gamma^{-1}\Phi\Gamma)=\Gamma^{-1}(\mathrm{Com \,}\Phi)\Gamma\quad\text{for all}\quad\Gamma\in\mathrm{GL}(n,\C).
 \end{align}
\end{defn}
\begin{lem}\label{19.1.16''}  $\mathrm{GCom\,} \Phi$  is connected, for each $\Phi\in\mathrm{Mat}(n\times n,\C)$.
 \end{lem}
\begin{proof}
Let $\Theta\in \mathrm{GCom\,} \Phi$. Since the set of eigenvalues of $\Phi$ is finite, and the numbers
$0$ and $-1-\Vert \Theta\Vert$ do not belong to it,  then we can find a continuous map
$\lambda:[0,1]\to \C$ such that $\lambda(0)=0$, $\lambda(1)= 1+\Vert\Theta\Vert$ and $\Theta+\lambda(t)I\in\mathrm{GL}(n,\C)$ for all $0\le t\le 1$.
Setting
\[
\gamma(t)=\begin{cases}\Theta+\lambda(t)I\quad&\text{if}\quad 0\le t\le 1,\\
(1+\Vert \Theta\Vert)\Big(\frac{2-t}{1+\Vert \Theta\Vert}\Theta+I\Big)\quad &\text{if}\quad 1\le t\le 2,\\
(1+(3-t)\Vert\Theta\Vert) I \quad &\text{if}\quad 2\le t\le 3,
\end{cases}
\]
then we obtain a continuous path $\gamma$ in $\mathrm{GL}(n,\C)$, which connects $\Theta=\gamma(0)$ with
$I=\gamma(3)$. Since $\Theta\in \mathrm{Com\,}\Phi$, from the definition of $\gamma$ it is clear that the values of $\gamma$ belong to the algebra $\mathrm{Com\,\Phi}$. In view of \eqref{17.8.16+}, this means that $\gamma$ lies inside $\mathrm{GCom\,}\Phi$.
\end{proof}

\begin{defn}\label{18.8.16n'}
Let $X$ be a complex space, and  $A:X\to \mathrm{Mat}(n\times n,\C)$ holomorphic. We introduce the families
\[
 \mathrm{Com\,}A:=\big\{\mathrm{Com\,}A(z)\big\}_{z\in X}\quad\text{and}\quad \mathrm{GCom\,}A:=\big\{\mathrm{GCom\,}A(z)\big\}_{z\in X}.
\]
If the dimension of  $\mathrm{Com\,}A(z)$ does not depend on $z$, then it follows from Lemma \ref{9.12.15} that $\mathrm{Com\,}A$ is a holomorphic vector bundle, but it is clear that this dimension can jump (in an analytic set).  But even if $\mathrm{Com\,}A$ is a holomorphic vector bundle,   $\mathrm{Com\,}A$ need not be locally trivial as a bundle of algebras. In particular, $\mathrm{GCom\,}A$ need not be locally trivial as a bundle of groups. Moreover, it is possible that  $\mathrm{GCom\,}A$ is not locally trivial as a bundle of topological spaces. Here is an example.
\end{defn}

\begin{leer}\label{24.8.16'}{\bf Example.}
Let $X=\C$ and $A(z):=\begin{pmatrix}z &1\\0&0\end{pmatrix}$, $z\in \C$.
Then
\[
\begin{pmatrix}a&b\\c&d\end{pmatrix}\in\mathrm{Com\,}A(z)\Leftrightarrow\begin{pmatrix}za+c &zb+d\\0&0\end{pmatrix}=\begin{pmatrix}za &a\\zc&c\end{pmatrix}\Leftrightarrow c=0\text{ and }a=zb+d,
\]which implies that $\dim \mathrm{Com\,}A(z)=2$ for all $z\in\C$. However
\[
\mathrm{GCom\,}A\,(0)=\bigg\{\begin{pmatrix}a&b\\0&a\end{pmatrix}\;\bigg\vert\; a\in \C^*, b\in \C\bigg\}
\]whereas, for $z\not=0$, $\mathrm{GCom\,}A\,(z)$ is isomorphic to
\[
\bigg\{\begin{pmatrix}a&0\\0&d\end{pmatrix}\;\bigg\vert\; a,d\in \C^*\bigg\}\quad\text{if}\quad z\not=0,
\]which implies that $\pi_1\big(\mathrm{GCom\,}A\,(0)\big)=\Z$ whereas
$\pi_1\big(\mathrm{GCom\,}A\,(z)\big)=\Z^2 $ if $ z\not=0$. Hence, for $z\not=0$, $\mathrm{GCom\,}A(z)$
is  not homeomorphic to $\mathrm{GCom\,}A(0)$.
\end{leer}

\begin{defn}\label{24.8.16''}  Let $X$ be a complex space, and  $A:X\to \mathrm{Mat}(n\times n,\C)$ holomorphic.
Even if the families $\mathrm{Com\,}A$ and/or $\mathrm{GCom\,}A$ are not locally trivial, their sheaves of holomorphic,  continuous and $\Cal C^\infty$ sections  are well-defined. We denote them by $\Cal O^{\mathrm{Com\,}A}$, $\Cal C^{\mathrm{Com\,}A}$, $(\Cal C^\infty)^{\mathrm{Com\,}A}$ etc.

Further, we define sheaves
$\widehat{\Cal C}^{\mathrm{Com\,}A}$  and $\widehat{\Cal C}^{\mathrm{GCom\,}A}$ as follows: if  $U$ is a non-empty open subset of $X$, then   $\widehat{\Cal C}^{\mathrm{Com\,}A}(U)$ is the algebra of all $f\in\Cal C^{\mathrm{Com\,}A}(U)$
such that, for each $\xi\in U$, the following condition is satisfied:
\begin{equation}\label{19.8.16}\begin{cases}&\text{there exist a neighborhood }V\text{ of }\xi\\&\text{and }h\in \Cal O^{\mathrm{Com\,}A}(V)\text{ such that }h(\xi)=f(\xi)
\end{cases}\end{equation}and we set $\widehat{\Cal C}^{\mathrm{GCom\,}A}(U)=\Cal C^{\mathrm{GCom\,}A}(U)\cap\widehat{\Cal C}^{\mathrm{Com\,}A}(U)$.

{\em Remark 1:} By Spallek's criterion (Theorem \ref{28.6.16**}, condition (iii)), $(\Cal C^\infty)^{\mathrm{Com\,}A}$ and $(\Cal C^\infty)^{\mathrm{GCom\,}A}$ are subsheaves of $\widehat{\Cal C}^{\mathrm{Com\,}A}$ and $(\widehat{\Cal C}^\infty)^{\mathrm{GCom\,}A}$, respectively.

{\em Remark 2:} If $X$ is a Riemann surface, then, by the Smith criterion (Theorem \ref{28.6.16**}, condition (ii)), $\widehat{\Cal C}^{\mathrm{Com\,}A}={\Cal C}^{\mathrm{Com\,}A}$ and $\widehat{\Cal C}^{\mathrm{GCom\,}A}={\Cal C}^{\mathrm{GCom\,}A}$.
\end{defn}

The following Oka principle is a special case of a result  O. Forster and K. J.  Ramspott \cite[Satz 1]{FR1}.
\begin{prop}\label{28.6.16-} Let $X$ be a Stein space, and $A:X\to \mathrm{Mat}(n\times n,\C)$ a holomorphic map. Then each $\widehat{\Cal C}^{\mathrm{GCom\,}A}$-trivial $\Cal O^{\mathrm{GCom\,}A}$-cocycle is ${\Cal O}^{\mathrm{GCom\,}A}$-trivial.
\end{prop}
Indeed it is easy too see that, for each non-empty open $U\subseteq X$ we have: If $h\in \Cal O^{\mathrm{Com\,}A}(U)$, then $e^h\in\Cal O^{\mathrm{GCom\,}A}(U)$, and, if $ H\in\Cal O^{\mathrm{GCom\,}A}(U)$  with $\sup_{\zeta\in U}\Vert H(\zeta)-I\Vert<1 $, then
\[
\log H:=\sum_{\mu=1}^\infty (-1)^{\mu-1}\frac{(H-I)^\mu}{\mu}\;\in\; \Cal O^{\mathrm{Com\,}A}(U).
\]
This shows that $\Cal O^{\mathrm{GCom\,}A}$ is a {\em coherent $\Cal O$-subsheaf} of $\Cal O^{\mathrm{GL}(n,\C)}_X$ in the sense of \cite[\S 2]{FR1}, where $\Cal O^{\mathrm{Com\,}A}$ is the {\em generating sheaf of Lie algebras}.
Moreover, as observed in \cite[\S 2.3, example b)]{FR1}), the pair $\big(\Cal O^{\mathrm{GCom\,}A},\widehat{\Cal C}^{\mathrm{GCom\,}A}\big)$ is an {\em admissible pair} in the sense of \cite{FR1}, which, trivially, satisfies condition (PH) in Satz 1 of \cite{FR1}). Therefore, by that Satz 1,
each $\widehat{\Cal C}^{\mathrm{GCom\,}A}$-trivial $\Cal O^{\mathrm{GCom\,}A}$-cocycle is ${\Cal O}^{\mathrm{GCom\,}A}$-trivial. \qed

\begin{leer}\label{24.8.16*}{\em Proof of Theorem \ref{24.8.16+}:} Since $A$ and $B$ are locally holomorphically similar at each point of $X$, we can find an open covering $\{U_i\}_{i\in I}$ of $X$ and holomorphic maps $H_i:U_i\to \mathrm{GL}(n,\C)$ such that
\[
B=H_i^{-1}AH_i^{}\quad\text{on}\quad U_i.
\]Then
$H_i^{-1}AH_i^{}=B=H_j^{-1}AH_j^{}$ on $U_i\cap U_j$. Hence $AH_i^{}H_j^{-1}=H_i^{}H_j^{-1}A$ on $U_i\cap U_j$, i.e.,
 the family
 \begin{equation}\label{24.8.16**}\{H_i^{}H^{-1}_j\}_{i,j\in I}
 \end{equation} is an $\Cal O^{\mathrm{GCom\,}A}$-cocycle.

Now, by hypothesis, we have a $\Cal C^\infty$ map $T:X\to \mathrm{GL}(n,\C)$ with $B=T^{-1}AT$ on $X$. Set $c_i=H_iT$ on $U_i$. Then
\[
c_iA=H_iTA=H_iBT=AH_iT=Ah_i,
\]i.e., $h_i\in (\Cal C^\infty)^{\mathrm{Com\,}A}(U_i)$, and
\[
c^{}_ic_j^{-1}=H_i^{}TT^{-1}H_j^{-1}=H_i^{}H_j^{-1} \quad\text{on}\quad U_i\cap U_j.
\]Hence \eqref{24.8.16**} is $(\Cal C^\infty)^{\mathrm{Com\,}A}$-trivial. Since, by Spallek's criterion (Theorem \ref{28.6.16**} (condition (iii)), $(\Cal C^\infty)^{\mathrm{Com\,}A}$ is a subsheaf of $\widehat{\Cal C}^{\mathrm{Com\,}A}$, it follows that \eqref{24.8.16**} is $\widehat{\Cal C}^{\mathrm{Com\,}A}$-trivial. By Proposition \ref{28.6.16-} this means that it is $\Cal O^{\mathrm{GCom\,}A}$-trivial, i.e.,
\[
H_i^{}H_j^{-1}=h_i^{}h_j^{-1}\quad\text{on}\quad U_i\cap U_j,
\]for some family $h_i\in \Cal O^{\mathrm{GCom\,}A}(U_i)$. Then
\[h_i^{-1}H_i^{}=h_j^{-1}H_j^{}\quad\text{on}\quad U_i\cap U_j.\]Hence, there is a well-defined  holomorphic map $H:X\to \mathrm{GL}(n,\C)$ with
$H=H_i^{-1}h_i^{}$ on $U_i$,  and which satisfies $
H^{-1}BH=h_i^{-1}H_i^{}BH_i^{-1}h_i^{}=h_i^{-1}Ah_i^{}=A$. \qed
\end{leer}

\section{Reduction of Theorem \ref{28.6.16*}{}}\label{5.10.16}

Let $X$ be a one-dimensional Stein space.

First note that Theorem \ref{28.6.16*} follows by standard arguments from the following

\begin{thm}\label{23.6.16'} For each holomorphic map $A:X\to \mathrm{Mat}(n\times n,\C)$,
\[ H^1(X,\Cal O^{\mathrm{GCom\,}A})=0.
\]
\end{thm}

Indeed, assume Theorem \ref{23.6.16'} is already proved and $A, B$ are as in Theroem \ref{28.6.16*}. Then we can find an open covering $\{U_i\}_{i\in I}$ of $X$ and holomorphic  maps $H_i:U_i\to\mathrm{GL}(n,\C)$, $i\in I$, such that
\begin{equation}\label{28.6.16}B=H_i^{-1}AH_i^{}\quad\text{on}\quad U_i.
\end{equation} It follows that $AH_i^{}H_j^{-1}=H_i^{}H_j^{-1}A$ on $U_i\cap U_j$. Hence, $\{H_i^{}H_j^{-1}\}^{}_{i,j\in I}$ is an  $\Cal O^{\mathrm{GCom\,}A}$-cocycle. By Theorem \ref{23.6.16'}, this cocycle is $\Cal O^{\mathrm{GCom\,}A}$-trivial, i.e., there is a family $h_i\in\Cal O^{\mathrm{GCom\,}A}(U_i)$ with $H_i^{}H_j^{-1}=h_i^{}h_j^{-1}$ on $U_i\cap U_j$. Therefore $h_i^{-1}H_i^{}=h_j^{-1}H_j^{}$ on $U_i\cap U_j$.
Hence, there is a well-defined global holomorphic  map $H:X\to\mathrm{GL}(n,\C)$ with $H=h_i^{-1}H_i^{}$ on $U_i$ for all $i\in I$.
 From  \eqref{28.6.16} and the relations $h_i^{}Ah_i^{-1}=A$ it follows  that $H^{-1}AH=B$ on $X$.
\qed

To prove Theorem \ref{23.6.16'}, by Proposition \ref{28.6.16-}, it is sufficient to prove that each ${\Cal O}^{\mathrm{GCom\,}A}$-cocycle is $\widehat{\Cal C}^{\mathrm{GCom\,}A}$-trivial. Sections 6 - 9 are devoted to the proof of this fact, as a special case of the following stronger result.

\begin{thm}\label{4.10.16} For each holomorphic map $A:X\to \mathrm{Mat}(n\times n,\C)$,
\[ H^1(X,\widehat{\Cal C}^{\mathrm{GCom\,}A})=0.
\]
\end{thm}

\begin{rem}\label{3.9.16**} Theorem \ref{23.6.16'} contains the statement $H^1(X,\Cal O^{\mathrm{GCom\,}\Phi})=0$ for each matrix $\Phi\in \mathrm{Mat}(n\times n,\C)$ and each one-dimensional Stein space $X$.  Since
$\mathrm{GCom\,}(\Phi)$ is connected (Lemma \ref{19.1.16''}), this is a special case of the statement
\begin{equation}\label{27.8.16'}
H^1(X,\Cal O^G)=0
\end{equation}for each connected complex Lie group $G$ and each one-dimensional Stein space $X$.
If $X$ is smooth, \eqref{27.8.16'} was proved by H. Grauert \cite[Satz 7]{Gr}. For non-smooth $X$, surprisingly, it seems that there is no explicite  reference for \eqref{27.8.16'} in the literature, except for $G=\mathrm{GL}(n,\C)$, see \cite[Theorem 7.3.1 (c) or Corollary 7.3.2 (1)]{Fc}. Therefore I asked colleagues and got two answers.

F. Forstneri\v c answered  that, by \cite{H}, each one-dimensional Stein space has the homotopy type of a one-dimensional CW complex and, therefore,
\begin{equation}\label{27.8.16''}H^1(X,\Cal C^G)=0,
\end{equation} which then implies \eqref{27.8.16'} by Grauert's Oka priciple \cite[Satz I]{Gr} (see also \cite[7.2.1]{Fc}).

J. Ruppenthal proposed to pass to the normalization of $X$,  which is smooth. At least if $X$ is irreducible and, hence, homeomorphic to its normalization, this immediately reduces the topological statement \eqref{27.8.16''} to the smooth case, which then  implies \eqref{27.8.16'} by Grauert's Oka principle. This idea will be used  in the proof of Theorem \ref{17.8.16--} below.
\end{rem}

\section{Jordan stable points}\label{8.6.16}

\begin{defn}\label{8.6.16n'}
Let  $X$ be a complex space, and $A:X\to \mathrm{Mat}(n\times n,\C)$ a holomorphic map.
A point $\xi\in X$ will be called {\bf Jordan stable} for $A$ if there exist a neighborhood $U$ of $\xi$ and holomorphic functions $\lambda_1,\ldots,\lambda_m:U\to \C$ such that

-- for each $\zeta\in U$, the values $\lambda_1(\zeta),\ldots, \lambda_m(\zeta)$ are the eigenvalues of $A(\zeta)$ and $\lambda_i(\zeta)\not=\lambda_j(\zeta)$ for all $1\le i,j\le m$ with $i\not=j$;

-- for each $1\le j\le m$ and each $1\le k\le n$, the number of Jordan blocks of size $k$ of $\lambda_j(\zeta) $ as an eigenvalue of $A(\zeta)$ is the same for all $\zeta\in U$.
\end{defn}

The following result is known and important for the purpose of the present paper.
\begin{prop}\label{11.6.16} If $X$ is a Riemann surface, then, for each holomorphic $A:X\to \mathrm{Mat}(n\times n,\C)$,  the points in $X$ which are not Jordan stable for $A$ form a discrete and closed subset of $X$.\footnote{Actually,  the following more general fact is true: For each complex space and each holomorphic $A:X\to \mathrm{Mat}(n\times n,\C)$,  the points in $X$ which are not Jordan stable for $A$ form a nowhere dense analytic subset of $X$. I have a proof but no reference for this.  In the smooth case, in \cite[S 3.4]{B3} one can find the statement that this set is {\em contained} in a  nowhere dense analytic subset of $X$. In the  present paper, we do not use this generalization. }
\end{prop}
 This was proved by H. Baumg\"artel \cite[Ch.V, \S 7]{B1}.

\begin{lem}\label{26.8.16} Let $\Phi, \widetilde\Phi,\Psi,\widetilde\Psi\in \mathrm{Mat}(n\times n,\C)$ such that $\Phi$ and $\widetilde\Phi$ are similar, and $\Psi$ and $\widetilde\Psi$ are similar. Then
\begin{equation*}
\dim\Big\{\Theta\in\mathrm{Mat}(n\times n,\C)\;\Big\vert\;\Phi\Theta=\Theta\Psi\Big\}=\dim\Big\{\Theta\in\mathrm{Mat}(n\times n,\C)\;\Big\vert\;\widetilde\Phi\Theta=\Theta\widetilde\Psi\Big\}.
\end{equation*}
\end{lem}
\begin{proof}Since $\Phi$ and $\widetilde\Phi$ are similar, we have $\Gamma_\Phi,\Gamma_\Psi\in\mathrm{GL}(n,\C)$ with $\Phi=\Gamma^{-1}_\Phi\widetilde \Phi\Gamma_\Phi^{}$ and $\Psi=\Gamma^{-1}_\Psi\widetilde \Psi\Gamma_\Psi^{}$. Then, for each $\Theta\in\mathrm{Mat}(n\times n,\C)$,
\[
\Phi\Theta=\Theta\Psi\Longleftrightarrow \Gamma^{-1}_\Phi\widetilde \Phi\Gamma^{}_\Phi\Theta=\Theta\Gamma^{-1}_\Psi\widetilde\Psi\Gamma^{}_\Psi\Longleftrightarrow\widetilde \Phi\Gamma_\Phi^{} \Theta\Gamma_\Psi^{-1} = \Gamma_\Phi^{}\Theta\Gamma_\Psi^{-1}\widetilde\Psi,
\]which means that the map $\Theta\mapsto \Gamma_\Phi^{}\Theta\Gamma_\Psi^{-1}$ is a linear isomorphism from

\noindent $\Big\{\Theta\in\mathrm{Mat}(n\times n,\C)\;\Big\vert\;\Phi\Theta=\Theta\Psi\Big\}$ onto $\Big\{\Theta\in\mathrm{Mat}(n\times n,\C)\;\Big\vert\;\widetilde\Phi\Theta=\Theta\widetilde\Psi\Big\}$.
Hence the dimension of these spaces are equal.
\end{proof}

\begin{lem}\label{26.1.16''}  Let  $X$ be a complex space, $A:X\to \mathrm{Mat}(n\times n,\C)$  a holomorphic map, and
$\xi\in X$ a Jordan stable point  of $A$. Then there exist a neighborhood $U$ of $\xi$ and a holomorphic map $H:U\to GL(n,\C)$ such that (Def. \ref{6.6.16''})
\begin{equation}\label{27.1.16--}
H(\zeta)^{-1}\big(\mathrm{Com\,}A(\zeta)\big)H(\zeta)=\mathrm{Com\,}A(\xi)\quad\text{for all}\quad \zeta\in U.
\end{equation}
\end{lem}
\begin{proof}
Let $U$ and $\lambda_j$ be as in Definition \ref{8.6.16n'}. By hypothesis, for each $1\le j\le m$ and each $1\le \ell\le m$, the number of Jordan blocks of size $\ell$ of $\lambda_j(\zeta)$ does not depend on $\zeta\in U$; we denote it by $\beta(j,\ell)$ ($=0$ if there is no such block). Set
\begin{align*}&k_j=\sum_{\ell=1}^n\beta(j,\ell) \quad(=\text{ the algebraic multiplicity of }\lambda_j(\zeta)\text{ for all }\zeta\in U)\quad\text{and}\\
&N_{\ell}=\big(\delta_{\mu,\nu-1}\big)_{\mu,\nu=1}^{\ell}\quad(\delta_{\mu\nu}=\;\text{Kronecker symbol},\;\mu=\;\text{row index}),
\end{align*}and let $M_j$ be a $k_j\times k_j$ matrix which can be written as a block diagonal where, for each $1\le\ell\le n$ with $\beta(j,\ell)\not=0$, precisely $\beta(j,\ell)$ of the blocks on the diagonal are equal to $N_{\ell}$.
Denote by $J(\zeta)$  the $n\times n$ matrix defined by the block diagonal matrix with
\[
\lambda_1(\zeta)I_{k_1}+M_{1}\;,\;\ldots\;,\;\lambda_m(\zeta)I_{k_m}+M_{m}
\]on the diagonals.
Then, for each $\zeta\in U$, $J(\zeta)$ is a Jordan normal form of $A(\zeta)$. Let us fix some  $\Phi\in \mathrm{GL}(n,\C) $ with
\begin{equation}\label{27.8.16neu}A(\xi)=\Phi J(\xi)\Phi^{-1}.
\end{equation}
Since $\lambda_i(\zeta)\not=\lambda_j(\zeta)$ if $i\not=j$,  a matrix $\Theta\in\mathrm{Mat}(n\times n,\C)$ belongs to $\mathrm{Com\,} J(\zeta)$ if and only if $\Theta$ is a block diagonal matrix with matrices
\[Z_1\in \mathrm{Com\,}\big(\lambda_1(\zeta)I_{k_1}+M_{1}\big)\;,\;\ldots\;,\;Z_m\in \mathrm{Com\,}\big(\lambda_m(\zeta)I_{k_m}+M_{m}\big)\]
on the diagonal \cite[Ch. VIII, \S 1]{Ga}. Since $\mathrm{Com\,}\big(\lambda_j(\zeta)I_{k_j}+M_{j}\big)=\mathrm{Com\,}M_{j}$, this in particular shows that
\begin{equation}\label{26.8.16''}
\mathrm{Com\,}J(\zeta)=\mathrm{Com\,}J(\xi)\quad\text{for all}\quad \zeta\in U.
\end{equation}
Since $A(\zeta)$ and $J(\zeta)$ are similar for all $\zeta\in U$, this and  Lemma \ref{26.8.16} imply that
\[
\dim\Big\{\Theta\in\mathrm{Mat}(n\times n,\C)\;\Big\vert\;A(\zeta)\Theta=\Theta J(\zeta)\Big\}=\dim\mathrm{Com\,}J(\zeta)=\dim\mathrm{Com\,}J(\xi)
\]does not depend on $\zeta\in U$. Therefore, by
\eqref{27.8.16neu} and by Theorem \ref{28.6.16**} (condition (i)),  after shrinking $U$, we can find a holomorphic $h:U\to \mathrm{Mat}(n\times n,\C)$ with $h(\xi)= \Phi$ and $A(\zeta)h(\zeta)=h(\zeta)J(\zeta)$ for all $\zeta\in U$. Since $\Phi\in \mathrm{GL}(n,\C)$, by a further shrinking of $U$, we can achieve that $h(\zeta)\in \mathrm{GL}(n,\C)$ for all $\zeta\in U$. Then $h(\zeta)^{-1}A(\zeta)h(\zeta)=J(\zeta)$ for all $\zeta\in U$ and, by \eqref{26.8.16+} and again by \eqref{26.8.16''},
\[
h(\zeta)^{-1}\big(\mathrm{Com\,}A(\zeta)\big)h(\zeta)=\mathrm{Com\,} J(\zeta)=\mathrm{Com\,} J(\xi).
\]Since, by \eqref{27.8.16neu} and \eqref{26.8.16+}, $\Phi\big(\mathrm{Com\,}J(\xi)\big)\Phi^{-1}=\mathrm{Com\,}A(\xi)$, now $H(\zeta):=h(\zeta)\Phi^{-1}$, $\zeta\in U$, has the required properties.
\end{proof}

\begin{thm}\label{23.7.16}Let  $X$ be a complex space, and $A:X\to \mathrm{Mat}(n\times n,\C)$ holomorphic. Let $W\subseteq X$ be an open set, which is contractible and such that all points of $W$ are Jordan stable for $A$, and let $\xi\in W$. Then there exists a continuous map
$T:W\to \mathrm{GL}(n,\C)$  such that (Def. \ref{6.6.16''})
\begin{equation}\label{27.1.16--}
T^{-1}(\zeta)\big(\mathrm{Com\,}A(\zeta)\big)T(\zeta)=\mathrm{Com\,}A(\xi)\quad\text{for all}\quad \zeta\in W.
\end{equation} If, moreover, $W$ is Stein, then this $T$ can be chosen holomorphically on $W$.
\end{thm}
\begin{proof} Since $W$ is connected, by Lemma \ref{26.1.16''}, for each $\eta\in W$, we can find a neighborhood $U_\eta\subseteq W$ of $\eta$
and a holomorphic map $H_\eta:U_\eta\to GL(n,\C)$ such that
\begin{equation}\label{27.8.16}H_\eta^{-1}(\zeta)\big(\mathrm{Com\,}A(\zeta)\big)H_\eta^{}(\zeta)=\mathrm{Com\,}A(\xi)\quad\text{for all}\quad\zeta\in U_\eta.
\end{equation}
Denote by $G$ the normalizer of $\mathrm{GCom\,}A(\xi)$ in $\mathrm{GL}(n,\C)$, i.e., the complex Lie group of all $\Phi\in\mathrm{GL}(n,\C)$ with
$\Phi^{-1}\big(\mathrm{GCom\,}A(\xi)\big)\Phi=\mathrm{GCom\,}A(\xi)$.
Then, by \eqref{27.8.16},
\begin{equation*}
H^{-1}_\eta(\zeta)H^{}_\tau(\zeta)\in G\quad\text{for all}\quad \zeta\in U_\eta\cap U_\tau\quad\text{for all}\quad\eta,\tau\in W,
\end{equation*} i.e., the family $\{H^{-1}_\eta H^{}_\tau\}_{\eta\tau}$ is a $\Cal O^G$-cocycle with the covering $\{U_\eta\}_{\eta\in W}$. Since $W$ is contractible, this cocycle is $\Cal C^G$-trivial \cite[Corollary 11.6]{St},
i.e., there is a family $T_\eta\in \Cal C^G(U_\eta)$ such that $T^{-1}_\eta T^{}_\tau=H_\eta^{}H_\tau^{-1}$ on $U_\eta\cap U_\tau$. Then
\[
H_\eta^{}T^{}_\eta=H_\tau^{}T^{}_\tau\quad\text{on}\quad U_\eta\cap U_\tau,\quad \eta,\tau\in W.
\]Hence, there is a well-defined continuous map $T:X\to \mathrm{GL}(n,\C)$ with
\[
T=H_\eta^{}T^{}_\eta\quad\text{on}\quad U_\eta,\quad \eta\in W.
\]By \eqref{27.8.16} and since $T_\eta(\zeta)\in G$, this $T$ satisfies \eqref{27.1.16--}.

If $W$ is Stein, then by Grauert's Oka \cite[Satz 6]{Gr} the maps $T_\eta$ can be chosen to be holomorphic, which implies that  $T$ is holomorphic.
\end{proof}
\begin{cor}\label{9.8.16+}Under the hypotheses of Theorem \ref{23.7.16}, the group $\Cal C^{\mathrm{GCom\,} A}(W)$ (Def. \ref{6.6.16''}), endowed with   the topology of uniform convergence on the compact subsets of $X$, is connected.
\end{cor}
\begin{proof} Let $G=\mathrm{GCom\,} A(\xi)$, and let $\Cal C^{G}(W)$ be als also endowed with the topology of uniform convergence on the compact subsets of $W$.
 Since $W$ is contractible, each element of $\Cal C^{G}(X)$ can be connected by a continuous path in $\Cal C^{G}(X)$ with a constant map. Since $G$ is connected (Lemma \ref{19.1.16''}), this implies that $\Cal C^{G}(W)$ is connected. Since, by \eqref{27.1.16--},  $\Cal C^{\mathrm{GCom\,} A}(W)$ and $\Cal C^{G}(W)$ are isomorphic as topological groups, it follows that $\Cal C^{\mathrm{GCom\,} A}(W)$ is connected.
\end{proof}

\begin{rem}\label{17.3.16'}  In particular, the claim of Theorem \ref{23.7.16} is true if $W$ is  a simply connected, non-compact, connected Riemann surface. The question remains open whether the hypothesis on simple connectedness can be omitted?

If the normalizer of $\mathrm{GCom\,}A(\xi)$ is connected, the answer is affirmative. This follows by the same proof, using \cite[Satz 7]{Gr} (saying that $H^1(W,\Cal O^{G})=0$ for each connected complex Lie group $G$ and each non-compact connected Riemann surface $W$) intead of  \cite[Satz 6]{Gr}.
However, this is not always the case. For example,  the normalizer of
$\big(\begin{smallmatrix}1&0\\0&0\end{smallmatrix}\big)$ is
$\big\{\big(\begin{smallmatrix}a&0\\0&d\end{smallmatrix}\big)\;\big\vert\;a,d\in\C^*\big\}\cup \big\{\big(\begin{smallmatrix}0&b\\c&0\end{smallmatrix}\big)\;\big\vert\;b,c\in\C^*\big\}$, which
is not connected.
\end{rem}

\section{Bumps}\label{23.7.16+}

\begin{defn}\label{3.8.16} Let $X$ be a Riemann surface. Denote by $\Delta(r)$, $r>0$,  the closed disk  of radius $r$ centered at the origin in $\C$.
A pair $(E,F) $ will be called a {\bf bump in $X$} if $E$ and $F$ are compact subsets of $X$ such that

(i)  $E\cap F=\emptyset$, or

(ii) there exist $0<r<R<\infty$ and a $\Cal C^\infty$ diffeomorphism, $z$, from an open neighborhood of $F$  onto an open neighborhood of $\Delta(R)$ such that
\begin{align}&\label{9.8.16n}F\subseteq\{\vert z\vert\le R\},\\
&\label{9.8.16n'}E\cap \{\vert z\vert\le R\}\subseteq F,\\
&\label{9.8.16n''} F\cap\{r\le\vert z\vert\le R\}=E\cap \{r\le\vert z\vert\le R\}\not=\{r\le\vert z\vert\le R\}.
\end{align} Moreover, if $\Cal U$ is an open covering of $X$ and $K$ is a subset of $X$, then we say that $(E,F)$
\begin{itemize}
\item [--] {\bf is $\Cal U$-fine} if each $F$ is contained in at least one of the sets of $\Cal U$,
\item[--]  {\bf does not meet $K$} if  $K\cap F=\emptyset$.
\end{itemize}

We say that a  set $M\subseteq X$ is a {\bf bump extension} of $K\subseteq M$ if there exist finitely many bumps $(E_1,F_1),\ldots,(E_m,F_m)$ in $X$ such that $E_1=K$, $E_{\mu+1}=E_\mu\cup F_\mu$ for $1\le \mu\le m-1$ and $E_{m}\cup F_m=M$. Moreover, if $\Cal U$ is an open covering of $X$ and $K$ is a subset of $X$, then we shall say that this bump extension
\begin{itemize}
\item [--] {\bf is $\Cal U$-fine} if each $(E_\mu,F_\mu)$ is $\Cal U$-fine,
\item[--]  {\bf does not meet $K$} if  each $(E_\mu,F_\mu)$ does not meet $K$.
\end{itemize}
\end{defn}

\begin{lem}\label{7.8.16} Let $X$ be a Riemann surface, and let $\rho:X\to \R$ be a $\Cal C^\infty$ function such that, for some real numbers $\alpha<\beta$, the set $\{\rho\le \beta\}$ is compact and $\rho$ has no critical points on $\{\alpha\le \rho\le \beta\}$. Then, for each open covering $\Cal U$ of $X$, $\{\rho\le \beta\}$ is a bump extension of $\{\rho\le \alpha\}$, which is $\Cal U$-fine and does not meet $\{\rho\le \alpha-1\}$.
\end{lem}
\begin{proof} It is sufficient  to prove that, for each $\alpha\le t\le \beta$,
 there exists $\varepsilon>0$ such that, if $t-\varepsilon< t_1\le t\le t_2<t+\varepsilon$, then
$\{\rho\le t_2\}$ is a    bump extension of $\{\rho\le t_1\}$, which is $\Cal U$-fine and does not meet $\{\rho\le \alpha-1\}$.

Let $\alpha\le t\le \beta$ be given.

 Since $\{\rho=t\}$ is compact and $\rho$ has no critical points on $\{\rho=t\}$, then
we can find $\varepsilon_0>0$, open subsets $W_1,\ldots,W_m$ of $X$, $0<R<\infty$, and diffeomorphisms $z_\mu$ from $W_\mu$ onto a neighborhood of $\Delta(R)$  such that $\big\{t-\varepsilon_0\le \rho\le t+\varepsilon_0\}$ is compact and
\begin{itemize}
\item[(a)] $\rho$ has no critical points on $\{t-\varepsilon_0\le \rho\le t+\varepsilon_0\}$;
\item[(b)] for each $1\le \mu\le m$,  $\rho=\im z_\mu+t$ on $W_\mu$;
\item[(c)] $\big\{t-\varepsilon_0\le \rho\le t+\varepsilon_0\}\subseteq \{\vert z_1\vert<R/2\}\cup\ldots\cup \{\vert z_1\vert<R/2\}$;
\item[(d)] each $W_\mu$ is contained in $\{\rho>\alpha-1\}$ and in at least one of the sets of the covering $\Cal U$.
\end{itemize}By (c), we can find $\Cal C^\infty$-functions $\chi_1,\ldots,\chi_m:X\to [0,1]$ such that  $\supp\chi_\mu\subseteq \{\vert z_\mu\vert<R/2\}$ for $1\le \mu\le m$, and
$\sum_{\mu=1}^m\chi_\mu=1$ on $\big\{t-\varepsilon_0\le \rho\le t+\varepsilon_0\}$.
Set $\varepsilon=\min(\varepsilon_0,R/4)$. Then, by (b),
\begin{equation}\label{6.8.16}
\{\rho\le t+\varepsilon\}\cap\{R/2\le \vert z_\mu\vert\le R\}\subseteq \{\im z_\mu\le R/4\}\cap\{R/2\le \vert z_\mu\vert\le R\}.
\end{equation}

Let $t-\varepsilon<t_1\le t\le t_2<t+\varepsilon$ be given. We define
\begin{align*}&E_1=\big\{\rho\le t_1\big\},\\
& E_\mu=\Big\{\rho-(t_2-t_1)\sum_{\nu=1}^{\mu-1}\chi_\nu\le t_1\Big\} \quad\text{for}\quad 2\le \mu\le m+1,\\
& F_\mu=E_{\mu+1}\cap\{\vert z_\mu\vert\le R\}\quad\text{for}\quad 1\le \mu\le m.
\end{align*}

We first prove that, for  each $1\le \mu\le m$, $(E_\mu,F_\mu)$ is a bump in $X$. For that, it is sufficient to show that, for each $1\le \mu\le m$,
\begin{align}&\label{11.8.16}F_\mu\subseteq\{\vert z_\mu\vert\le R\},\\
&\label{11.8.16'}E_\mu\cap \{\vert z_\mu\vert\le R\}\subseteq F_\mu,\\
&\label{11.8.16''} F_\mu\cap\{R/2\le\vert z_\mu\vert\le 1\}=E_\mu\cap \{R/2\le\vert z_\mu\vert\le 1\}\not=\{R/2\le\vert z_\mu\vert\le R\}.
\end{align}By definition of $F_\mu$, \eqref{11.8.16} is trivial. Since $E_\mu\subseteq E_{\mu+1}$, also \eqref{11.8.16'} is clear from the definition of $F_\mu$. Since $\chi_\mu=0$ on $\{R/2\le \vert z_\mu\vert\le R\}$, we have
\[
E_\mu\cap\{R/2\le \vert z_\mu\vert\le R\}=E_{\mu+1}\cap\{R/2\le \vert z_\mu\vert\le R\}.
\] Since, by definition of $F_\mu$,
\[E_{\mu+1}\cap\{R/2\le \vert z_\mu\vert\le R\}=F_\mu \cap \{R/2\le \vert z_\mu\vert\le R\},\] this proves the ``$=$'' in \eqref{11.8.16''}.
Since $E_\mu\subseteq\{\rho\le t_2\}\subseteq \{\rho\le t+\varepsilon\}$, from \eqref{6.8.16} we see
\[
E_\mu\cap\{R/2\le \vert z_\mu\vert\le R\}\subseteq \{\im z_\mu\le R/4\}\cap \{R/2\le \vert z_\mu\vert\le R\}.
\]
In particular, we have  the ``$\not=$'' in \eqref{11.8.16''}.

Since $\chi_\mu=0$ outside $\{\vert z_\mu\vert\le R\}$, we have
\[
E_\mu\setminus \{\vert z_\mu\vert\le R\}=E_{\mu+1}\setminus \{\vert z_\mu\vert\le R\}, \quad 1\le \mu\le m.
\]
Since $F_\mu\setminus\{\vert z_\mu\vert\le 1\}=\emptyset$, this implies that
\[
(E_\mu\cup F_\mu)\setminus\{\vert z_\mu\vert\le R\}=E_{\mu+1}\setminus \{\vert z_\mu\vert\le R\}, \quad 1\le \mu\le m,
\]Moreover, the relations $F_\mu=E_{\mu+1}\cap\{\vert z_\mu\vert \le R\}$ and $E_\mu\subseteq E_{\mu+1}$ imply that
\begin{multline*}
\Big(E_\mu\cup F_\mu\Big)\cap\{\vert z_\mu\vert\le R\}=\Big(E_\mu\cap\{\vert z_\mu\vert\le R\}\Big) \cup \Big(E_{\mu+1}\cap\{\vert z_\mu\vert\le R\}\Big)\\=E_{\mu+1}\cap \{\vert z_\mu\vert\le R\},\quad 1\le \mu\le m.
\end{multline*}Together this yields $E_\mu\cup F_\mu=E_{\mu+1}$ for $1\le \mu\le m$. Hence, $E_{m+1}$ is a  bump extension of $E_1$, which is $\Cal U$-fine and does not meet $\{\rho\le\alpha-1\}$ (by (d)). Since $E_1=\{\rho\le t_1\}$ and $E_{m+1}=\{\rho\le t_2\}$, this completes the proof of the lemma.
\end{proof}

\begin{thm}\label{5.8.16} Let $X$ be a non-compact connected Riemann surface,  and  $\Cal U$  an open covering of $X$. Then there exists a sequence $(B_\mu)_{\mu\in\N}$ of compact subsets of $X$ such that
\begin{itemize}
\item[(a)] for each $\mu\in \N$, $B_\mu$ is contained in at least one of the sets  of  $\Cal U$;
\item[(b)] for each $\mu\in \N^*$, $(B_0\cup\ldots\cup B_{\mu-1},B_{\mu})$ is a bump in $X$;
\item[(c)] $X=\bigcup_{\mu\in\N}B_\mu$.
\item[(d)] for each compact set $K\subseteq X$, there exists $N(\mu)\in\N$ such that $B_\mu\cap K=\emptyset$ if $\mu\ge N(\mu)$.
\end{itemize}
\end{thm}

\begin{proof} $X$ is a Stein manifold (see, e.g., \cite[Corollary 26.8]{F}). Therefore (see, e.g., \cite[Theorem 5.1.9]{Ho}), there is a strictly subharmonic function $\rho:X\to\R$ such that, for each $\alpha\in \R$, the set $\{\rho\le \alpha\}$ is compact. By Morse theory (see, e.g., \cite[\S 7 and \S 19, Exercise 19]{GP}), we may assume that all critical points of $\rho$ are non-degenerate, and, for each $\alpha\in\R$, at most one critical point of $\rho$ lies on $\{\rho=\alpha\}$.

In particular, then there is precisely one point in $X$, $\xi_{\mathrm{min}}$, where $\rho$ assumes its absolute minimum and this minimum is strong. Therefore, we can find $\varepsilon_0>0$ such that $\{\rho\le \rho(\xi_{\mathrm{min}})+\varepsilon_0\}$ is contained in at least one of the sets of $\Cal U$, and $\rho$ has no critical points in $\{\rho(\xi_{\mathrm{min}})<\rho\le \rho(\xi_{\mathrm{min}})+\varepsilon_0\}$.
Set $B_0=\{\rho\le \rho(\xi_{\mathrm{min}})+\varepsilon_0/2\}$. Then, by Lemma \ref{7.8.16}, for each $\varepsilon_0/2<\varepsilon\le \varepsilon_0$, the set $\{\rho\le \varepsilon\}$ is a $\Cal U$-fine bump extension of $B_0$. Therefore it sufficient to prove that, for each
 $t\ge\rho(\xi_{\mathrm{min}})+\varepsilon_0 $, the following is true.
 \begin{itemize}
\item [(*)] There exists $\varepsilon>0$ such that, if  $t-\varepsilon<t_1< t<t_2<t+\varepsilon$, then
$\{\rho\le t_2\}$ is a  bump extension of $\{\rho\le t_1\}$, which is $\Cal U$-fine and does not meet $\{\rho\le t_1-1\}$.
\end{itemize}

Let $t\ge\rho(\xi_{\mathrm{min}})+\varepsilon_0 $ be given.

If $\rho$ has no critical points on $\{\rho=t\}$, (*) follows from Lemma \ref{7.8.16}.

It remains  the case when precisely one critical point of $\rho$, $\xi$, lies on $\{\rho=t\}$.

Since the critical points of $\rho$ are non-degenerate and $\rho$ is strictly subharmonic (which yields that $\rho$ has no local maxima), then  either $\xi$ is the point of a strong local minimum of $\rho$, or $\xi$ is a saddle point of $\rho$.

First assume that $\rho$ has a  local minimum at $\xi$.
Then  $\xi$ is an isolated point of $\{\rho\le t\}$. Therefore we can find an open neighborhood $U$ of $\xi$ and an open neighborhood $V$ of $\{\rho\le t\}\setminus\{\xi\}$ such that $U\cap V=\emptyset$. Choose $\varepsilon>0$ such that
\[
\{\rho\le t+\varepsilon\}\subseteq U\cup V,
\]and
$U\cap \{\rho\le t+\varepsilon\}$ is contained in at least one set of $\Cal U$. Since $\rho$ has no critical points on $V\cap \{\rho=t\}$, we may moreover assume that $\rho$ has no critical points on $V\cap \{t-\varepsilon\le \rho\le t+\varepsilon\}$.

Let $t-\varepsilon<t_1< t< t_2< t+\varepsilon$ be given. Then, by Lemma \ref{7.8.16},  $V\cap \{\rho\le t_2\}$ is a  bump extension of $ \{\rho\le t_1\}$ ($=V\cap \{\rho\le t_1\}$), which is $\Cal U$-fine  and does not meet $\{\rho\le t_1-1\}$.
 Moreover, then
the pair $\big(V\cap\{\rho\le t_2\},U\cap \{\rho\le t+t_2\}\big)$ is a bump in $X$ (condition (i) in Definition \ref{3.8.16} ia satisfied), which is
$\Cal U$-fine  (as $U$ is contained in at least one of the sets of $\Cal U$) and does not meet $V\supseteq\{\rho\le t_1-1\}$, and such that $\big(V\cap\{\rho\le t_2\}\big)\cup\big(U\cap \{\rho\le t+t_2\}\big)=\{\rho\le t_2\}$. Together this completes the proof of  (*) if $\rho$ has a local minimum at $\xi$.

Now let $\xi$ be a saddle point of $\rho$. Since $\xi$ is the only critical point of $\rho$ on $\{\rho=t\}$ and the set of all critical points of $\rho$ is discrete and closed in $X$ (as all critical points of $\rho$ are non-degenerate), then we can find $\varepsilon>0$ such that $\xi$ is the only critical point of $\rho$ also in $\{t-\varepsilon\le \rho\le t+\varepsilon\}$.
Since the critical points of $\rho$ are non-degenerate,  by a lemma of Morse (see, e.g., \cite[Lemma 2.2]{M1}), we can find a neighborhood $W$ of $\xi$ and a $\Cal C^\infty$ diffeomorphism $z$ from $W$ onto a neighborhood of the origin in $\C$  such that $z(\xi)=0$ and, with $x:=\rea z$ and $y=\im z$,
\begin{equation}\label{13.8.16'}\rho=t+x^2-y^2\quad\text{on}\quad W.
\end{equation}Moreover, we may choose $W$ so small that
\begin{equation}\label{13.8.16''}
W\cap\{\rho\le t-1\}=\emptyset\quad\text{and } W\text{ is contained in at least one set of }\Cal U.
\end{equation}
Let $t-\varepsilon<t_1< t< t_2<t+\varepsilon$ be given. Choose $R>r>0$ such that $\{\vert z\vert\le R\}$ is compact and
\begin{equation}\label{12.8.16}\{\vert z\vert\le r\}\subseteq \{t_1<\rho<t_2\}.
\end{equation}
Further we choose a $\Cal C^\infty$ function $\chi:X\to [0,1]$ with $\chi=1$ on $\{\vert z\vert\le r/4\}$ and $\chi=0$ on $X\setminus\{\vert z\vert<r/2\}$. Since $\rho=t+x^2-y^2$ on $\{\vert z\vert\le R\}$, then we can find  $\delta>0$ so small that the functions
 \[
 \rho^{}_+:=\rho+\delta\chi\quad\text{and}\quad\rho^{}_-:=\rho-\delta\chi
 \]have the same critical points as $\rho$.

Since $\rho_+^{}=\rho^{}_-=\rho$ outside $\{\vert z\vert\le r\}$, $\rho_+^{}\ge \rho$ and $\rho_-^{}\le \rho$ on $\{\vert z\vert\le r\}$, and by \eqref{12.8.16}, we have
\begin{align}\label{12.8.16'}
&\{\rho^{}_+\le \tau\}=\{\rho\le \tau\}\quad\text{if}\quad \tau\le t_1,\\
\label{13.8.16}&\{\rho^{}_-\le \tau\}=\{\rho\le \tau\}\quad\text{if}\quad \tau\ge t_2.
\end{align}Since $\rho_+^{}(\xi)=\rho(\xi)+\delta=t+\delta>t $ and $\rho_-^{}(\xi)<t$, we have $\xi\not\in \{t_1\le \rho_+^{}\le t\}$
and $\xi\not\in \{t\le\rho_-^{}\le t_2\}$. Since $\xi$ is the only critical point of $\rho$ in $\{t-\varepsilon\le \rho\le t+\varepsilon\}$ and $\rho_\pm^{}$ have the same critical points as $\rho$, it follows that $\rho_+^{}$ has no critical points in $\{t_1\le \rho_+^{}\le t\}$, and $\rho_-^{}$ has no critical points in $\{t\le\rho_-^{}\le t_2\}$. By Lemma \ref{7.8.16} and \eqref{12.8.16'}, this shows that $\{\rho_+^{}\le t\}$ is a  bump extension of $\{\rho\le t_1\}$ ($=\{\rho_+^{}\le t_1\}$), which is $\Cal U$-fine and does not meet $\{\rho\le t_1-1\}$ ($=\{\rho_+^{}\le t_1-1\}$), and $\{\rho\le t_2\}$ ($=\{\rho_-^{}\le t_2\}$) is a  bump extension of $\{\rho_-^{}\le t\}$, which is $\Cal U$-fine and does not meet $\{\rho_-^{}\le t-1\}\supseteq \{\rho_-^{}\le t_1-1\}\supseteq \{\rho\le t_1-1\}$, where the last ``$\supseteq$'' holds, because $\rho_-^{}\le \rho$.

Therefore, it is now sufficient to find a bump in $X$, $(E,F)$, with
\begin{itemize}
\item[(I)] $E=\{\rho_+^{}\le t\}$,
\item[(II)] $E\cup F=\{\rho_-^{}\le t\}$,
\end{itemize} which is $\Cal U$-fine and does not meet $\{\rho\le t_1-1\}$.
For that we define $E$ by (I) and set
$F=\{\vert z\vert\le R\}\cap \{\rho^{}_-\le t\}$. Since $\rho_+^{}=\rho_-^{}$ outside $\{\vert z\vert\le R\}$, then also (II) holds.

Further, since $\rho_-^{}\le \rho^{}_+$ on $X$ and $\rho^{}_-=\rho^{}_+$ on $\{r\le\vert z\vert\le R\}$, we have \eqref{9.8.16n}, \eqref{9.8.16n'} and the ``$=$'' in \eqref{9.8.16n''}. Moreover, from \eqref{13.8.16'} we see
 \[
 E\cap\{r\le \vert z\vert \le R\}=\{\vert x\vert\le \vert y\vert\}\cap \{r\le \vert z\vert \le R\}\not=\{r\le \vert z\vert \le R\},
 \]which is the ``$\not=$'' in \eqref{9.8.16n''}. So, $(E,F)$ is a bump in $X$. From \eqref{13.8.16''} it follows that $(E,F)$ is $\Cal U$-fine and does not meet $\{\rho\le t_1-1\}$.
\end{proof}

\section{A topological vanishing theorem on one-dimensional Stein spaces}

\begin{defn}\label{20.7.16} Let $X$ be a complex space, $Z$ a discrete and closed subset of $X$, and $A:X\to \mathrm{Mat}(n\times n,\C)$ holomorphic.
Then we denote by $\Cal C^{Z,\,\mathrm{GCom\,}A}$ the following subsheaf of  $\Cal C^{\mathrm{GCom\,}A}$: if
 $U\subseteq X$ is  open and non-empty, then
$\Cal C^{Z,\,\mathrm{GCom\,}A}(U)$ is the group of all $T\in \Cal C^{\mathrm{GCom\,}A}(U)$ with $T=I$ in a neighborhood of  $Z\cap U$.
\end{defn}
\begin{lem}\label{23.5.16'} Let $X$ be a Riemann surface,  $Z$  a discrete and closed subset of $X$,  and $(E,F)$ a bump in $X$ with $E\cap F\not=\emptyset$ (Def. \ref{3.8.16}). Let $A:X\to \mathrm{Mat}(n\times n,\C)$ be  holomorphic, $U$ a neighborhood of $E\cap F$ and $f\in \Cal C^{Z,\,\mathrm{GCom\,}A}(U)$. Then there exist neighborhoods $U_E$ and $U_F$ of $E$ and $F$, respectively, and maps $f_E\in\Cal C^{Z,\,\mathrm{GCom\,}A}(U_E)$ and $f_F\in\Cal C^{Z,\,\mathrm{GCom\,}A}(U_F)$ such that $ U_E\cap U_F\subseteq U$
\begin{align}&\label{9.8.16} f=f_E^{}f_F^{-1}\quad\text{on}\quad U_E\cap U_F,\quad\text{and}\\
&\label{9.8.16'} f^{}_E=I\quad\text{on}\quad U^{}_E\setminus U.
\end{align}
\end{lem}
\begin{proof} Let $Y$ be the set of all points in $X$ which are not Jordan stable for $A$ (Def. \ref{8.6.16n'}). Since $Y$ is discrete and closed in $X$ (Proposition \ref{11.6.16}), then also $Z\cup Y$ is discrete and closed in $X$.

 By definition of a bump,  we have $R>r>0$, an open $\Omega\subseteq X$  and a $\Cal C^\infty$ diffeomorphism $z$ from $\Omega$ onto an open  neighborhood of $\Delta(R)$ with \eqref{9.8.16n} - \eqref{9.8.16n''}.
 Since $Z\cup Y$ is discrete and closed in $X$, we can  find $r\le s_1<s_2\le R$ such that
\begin{equation}\label{9.8.16*}(Z\cup Y)\cap \{s_1< \vert z\vert< s_2\}=\emptyset.
\end{equation}
Choose $s_1<s_0<s_2$ and let
\begin{equation}\label{10.8.16}\Gamma:=E\cap\{\vert z\vert=s_0\}=F\cap\{\vert z\vert=s_0\},
\end{equation}where the second ``$=$'' holds by the first ``$=$'' in \eqref{9.8.16n''}. Since $\Gamma$ is a closed subset of the circle $\{\vert z\vert=s_0\}$ (by definition) which is not equal to $\{\vert z\vert=s_0\}$ (by the second part of \eqref{9.8.16n''}), and since, by \eqref{10.8.16}, $\Gamma\subseteq E\cap F\subseteq U$, we can find finitely many open subsets $V''_j$, $V'_j$ and $V_j$, $1\le j\le m$, of $X$ such that
\begin{itemize}
\item[(a)] $\Gamma\subseteq V''_1\cup \ldots\cup V_m''$,
\item[(b)] $\overline {V''_j}\subseteq V'_j$ and $\overline {V_j'}\subseteq V_j$  for $1\le j\le m$,
\item[(c)] $V_j\subseteq \{s_1<\vert z\vert<s_2\}$ for $1\le j\le m$,
\item[(d)] $V_j\subseteq U$ for $1\le j\le m$,
\item[(e)] $V_1,\ldots, V_m$ are pairwise disjoint sets, each of which is contractible.
\end{itemize} Set $V''=V''_1\cup\ldots\cup V''_m$, $V'=V'_1\cup\ldots\cup V'_m$ and $V=V_1\cup\ldots\cup V_m$. Then, by (c) and \eqref{9.8.16*}, $V\cap Y=\emptyset$, which  implies, by (e) and Corollary \ref{9.8.16+}, that the group $\Cal C^{\mathrm{GCom\,}A}(V)$, endowed with the topology of uniform convergence on the compact subsets of $V$, is  connected. As $V\subseteq U$ (by (d)), it follows that there is a continuous map $\theta:[0,1]\to \Cal C^{\mathrm{GCom\,}A}(V)$ with
$\theta(0)=f\vert_V$ and $\theta(1)=I$ on $V$. Take $0=\tau_1<\tau_2<\ldots<\tau_m=1$ such that, with $g_j(\zeta):=\theta(\tau_j)(\zeta)\big(\theta(\tau_{j+1})(\zeta)\big)^{-1}-I$,
\[
\Vert g_j(\zeta)\Vert<1\quad\text{for all}\quad \zeta\in \overline{V'}\quad\text{and}\quad 1\le j\le m-1.
\] Choose a $\Cal C^\infty$ function $\chi:X\to [0,1]$ such that $\chi=1$ on $V''_j$ and $\chi=0$ outside $V'$, and define
\[
\widetilde f(\zeta)=\begin{cases}\big(I+\chi(\zeta)g_1(\zeta)\big)\cdot\ldots\cdot\big(I+\chi(\zeta)g_m(\zeta)\big)\quad&\text{if}\quad \zeta\in V,\\
I&\text{if}\quad \zeta\in X\setminus V.
\end{cases}
\]
Then  $\widetilde f=f$ on $V''$, $\widetilde f=I$ outside $V'$ and $\widetilde f\in \Cal C^{\mathrm{GCom\,}A}(X)$. Since, by (b), (c) and \eqref{9.8.16*}, $X\setminus \overline{V'}$ is a neighborhood of $Z$, it follows that even $\widetilde f\in \Cal C^{\mathrm{GCom\,}A, Z}(X)$.

Since, by (a), $V''$ is a neighborhood of $\Gamma$ and by \eqref{10.8.16}, we can find $s_1<t_1<s_0<t_2<s_2$ and  neighborhoods $U_E$ and $U_F$ of $E$ and $F$, respectively, such that
\begin{equation*}
(U_E\cup U_F)\cap \{t_1\le \vert z\vert\le t_2\}\subseteq V'',
\end{equation*}and hence
\begin{equation}\label{14.8.16} \widetilde f=f\quad\text{on}\quad (U_E\cup U_F)\cap \{t_1\le \vert z\vert\le t_2\}.
\end{equation}
By \eqref{9.8.16n} and  the ``$=$'' in \eqref{9.8.16n''}, we have
\[
F\setminus \{\vert z\vert<r\}=F\cap \{r\le \vert z\vert \le R\}=E\cap F\cap \{r\le \vert z\vert \le R\}\subseteq E\cap R\subseteq U.
\]Therefore, shrinking $U_F$, we can achieve that $U_F\setminus\{\vert z\vert<r\}\subseteq U$ and, hence,
\begin{equation}\label{15.8.16'}
U_F\setminus\{\vert z\vert<s_0\}\subseteq U.
\end{equation}By \eqref{9.8.16n'}, $E\cap\{\vert z\vert\le R\}\subseteq E\cap F$. Therefore, by  shrinking $U_E$, we can achieve that
$U_E\cap \{\vert z\vert\le R\}\subseteq U$ and, hence,
\begin{equation}\label{14.8.16'''}
U_E\cap \{\vert z\vert<s_0\}\subseteq U.
\end{equation}From \eqref{9.8.16n'} it follows that $E\cap \{vert z\vert\le R\}\subseteq E\cap F\subseteq U$. So, further shrinking $U_E$, we can achieve htat $U_E\cap \{vert z\vert R\}\subseteq U$ and, hence,
\begin{equation}\label{15.8.16--}
U_E\setminus U\subseteq U_E\setminus \{\vert z\vert\le R\}.
\end{equation}
By \eqref{14.8.16'''} and \eqref{15.8.16'}, the following definitions are correct
\[
f_E:=\begin{cases}  f\quad&\text{on}\quad U_E\cap \{\vert z\vert< s_0\},\\\widetilde f&\text{on}\quad U_E\setminus\{\vert z\vert< s_0\},\end{cases}\quad\text{and}\quad f_F:=\begin{cases}I\quad&\text{on}\quad U_F\cap \{\vert z\vert< s_0\},\\f^{-1}\widetilde f&\text{on}\quad U_F\setminus\{\vert z\vert< s_0\}.\end{cases}
\]By \eqref{14.8.16}, $f_E$ and $f_F$ are continuous on $U_E$ and $U_F$, respectively. Since  $\widetilde f\in \Cal C^{\mathrm{GCom\,}A, Z}(X)$ and $f\in \Cal C^{\mathrm{GCom\,}A, Z}(U)$, it follows that
$f_E\in \Cal C^{Z,\,\mathrm{GCom\,}A}(U_E)$ and $f_F\in \Cal C^{Z,\,\mathrm{GCom\,}A}(U_F)$. \eqref{9.8.16} is clear from the definition.
Moreover, by definition,
\begin{equation*}
f_E=\widetilde f\quad\text{on}\quad U_E\setminus \{\vert z\vert\le R\}.
\end{equation*}Since $\widetilde f=I$ outside $V'$ and $V'\subseteq \{\vert z\vert\le R\}$ (by (b) and (c)), this implies
\begin{equation*}
f_E=I\quad\text{on}\quad U_E\setminus \{\vert z\vert\le R\}.
\end{equation*} Together with \eqref{15.8.16--} this proves
\eqref{9.8.16'}.
\end{proof}

\begin{lem}\label{4.5.16}
 Let $X$ be a non-compact connected Riemann surface, $Z$ a discrete and closed subset of $X$, and $A:X\to \mathrm{Mat}(n\times n,\C)$ a holomorphic map. Then  $H^1\big(X,\Cal C^{Z,\,\mathrm{GCom\,}A}\big)=0$.
\end{lem}
\begin{proof} During this proof we use the abbreviation $\Cal F:=\Cal C^{Z,\,\mathrm{GCom\,}A}$, and, for each  compact set of $K\subseteq X$, we denote by $\Cal F(K)$ the group of germs of sections of $\Cal F$ in  neighborhoods of $K$.

Let an $\Cal F$-cocycle $f$ on $X$ be given.

Then, by Theorem \ref{5.8.16}, there exists a sequence $(B_\mu)_{\mu\in\N}$ of compact subsets of $X$ such that
\begin{itemize}
\item[(a)] each $B_\mu$ is contained in at least one set of the covering of $f$;
\item[(b)] for each $\mu\in \N$, $(B_0\cup\ldots\cup B_{\mu},B_{\mu+1})$ is a bump in $X$;
\item[(c)] $X=\bigcup_{\mu\in\N}B_\mu$.
\item[(d)] for each compact set $K\subseteq X$, there exists $N(\mu)\in\N$ such that $B_\mu\cap K=\emptyset$ if $\mu\ge N(\mu)$.
\end{itemize}
Since all $B_\mu$ are  compact, from (d)   we get $\mu_0\in\N^*$ and a sequence $\{L(\mu)\}_{\mu=\mu_0}^\infty $ of positive integers such that
\begin{align}&\label{15.8.16+}
\big(B_0\cup\ldots\cup B_{L(\mu)}\big)\cap B_\mu=\emptyset\quad\text{for}\quad \mu\ge \mu_0,\\
&\label{16.8.16'}
\lim_{\mu\to\infty} L(\mu)=\infty.
\end{align}
By (a), we can find a sequence $\{U_\mu\}_{\mu\in\N}$ of open subsets of $X$ such that, for each $\mu\in\N$,
$B_\mu\subseteq U_\mu$ and $U_\mu$ is contained in at least one of the sets of the covering of $f$. Then, by (c), $\Cal U:=\{U_\mu\}_{\mu\in\N}$ is an open covering of $X$ and we can find a $\Cal U$-cocycle $\widetilde f=\{\widetilde f_{\mu\nu}\}_{\mu,\nu\in\N}$ which is induced by $f$. By
Proposition \ref{17.12.15--}, it is sufficient to prove that $\widetilde f$ is $\Cal F$-trivial.

Let $g_{\mu\nu}$ the germ in $\Cal F(B_\mu\cap B_\nu)$ defined by $\widetilde f_{\mu\nu}$.

Then, again by Proposition \ref{17.12.15--}, it is sufficient to find a sequence $\{g_\mu\}_{\mu\in\N}$ of germs $g_\mu\in \Cal F(B_\mu)$ with $ g_{\mu\nu}=g_{\mu}^{}g^{-1}_\nu$ on $B_\mu\cap B_\nu$. This will be done if we have constructed  a sequence
of $k$-tuples $\{(g^{(k)}_0,\ldots,g^{(k)}_{k})\}_{\mu\in\N}$ of germs $g^{(k)}_\mu\in\Cal F(B_\mu)$ such that, for each $k\in \N$,
\begin{align}&\label{16.8.16}\big(g^{(k)}_{\mu}\big)^{-1}g_{\mu\nu}^{} g^{(k)}_\nu=I\quad\text{on}\quad B_\mu\cap B_\nu\quad\text{for all}\quad 0\le \mu,\nu\le k,\\
&\label{16.8.16''} \text{if }k\ge \mu_0,\text{ then }g^{(k)}_\mu=g^{(k-1)}_\mu\quad\text{for all}\quad 0\le \mu\le L(\mu),
\end{align}
because then  $g_\mu:=\lim_{k\to \infty }g^{(k)}_{\mu}$ are well-defined germs in $\Cal F(B_\mu)$ (by \eqref{16.8.16''} and \eqref{16.8.16'}), which have the required properties (by \eqref{16.8.16}).

To construct this sequence by induction, we set $g^{(0)}_0=I$ and assume that, for some $\ell\in\N$ and all $0\le k\le \ell$, we already have a  $k$-tuple $\{(g^{(k)}_0,\ldots,g^{(k)}_{k})\}_{\mu\in\N}$ of germs $g^{(k)}_\mu\in\Cal F(B_\mu)$, $0\le k\le \ell$, such that \eqref{16.8.16} and \eqref{16.8.16''} hold true for $0\le k\le \ell$.

Since $\widetilde f$ is a cocycle, we have $g^{}_{\mu\nu}g^{}_{\nu,\ell+1}=g^{}_{\mu,\ell+1}$ on $B_\mu\cap B_{\ell+1}$ for all $\mu,\nu$.  Since $g_{\mu\nu}^{}=g^{(k)}_{\mu}\big(g^{(k)}_\nu\big)^{-1}$ (as \eqref{16.8.16} holds for $k=\ell$), from this we obtain
\[
\big(g_\nu^{(\ell)}\big)^{-1}g_{\nu,\ell+1}^{}=\big(g_\mu^{(\ell)}\big)^{-1}g_{\mu,\ell+1}^{}\quad\text{on}\quad B_\mu\cap B_\nu\cap B_{\ell+1}\quad\text{for all}\quad 0\le \mu,\nu\le \ell.
\] Therefore, there is a well-defined germ $h\in\Cal F\big((B_0\cup\ldots\cup B_{\ell})\cap B_{\ell+1}\big)$ with
\begin{equation}\label{16.8.16-}
h=\big(g_\mu^{(\ell)}\big)^{-1}g_{\mu,\ell+1}^{}\quad\text{on}\quad B_\mu\cap B_{\ell+1}\quad\text{for}\quad 0\le \mu\le \ell.
\end{equation}
Since $(B_0\cup\ldots\cup B_\ell,B_{\ell+1}$ is a bump in $X$ with $\big(B_0\cup\ldots\cup B_{L(\ell+1)}\big)\cap B_{\ell+1}=\emptyset$, from Lemma \ref{23.5.16'} we get germs $h_{B_0\cup\ldots\cup B_\ell}^{}\in \Cal F\big(B_0\cup\ldots\cup B_\ell\big)$ and $h_{B_{\ell+1}}^{}\in \Cal F\big(B_{\ell+1}\big)$ such that
\begin{align}&\label{16.8.16++}
h^{-1}_{B_0\cup\ldots\cup B_\ell} h h^{}_{B_{\ell+1}}=I\quad\text{on}\quad B_0\cup\ldots\cup B_{L(\ell+1)},\\
&\label{16.8.16+++}\text{and, if }\ell+1\ge \mu_0, \text{ then }h_{B_0\cup\ldots\cup B_\ell}^{}=I\text{ on } B_0\cup\ldots\cup B_{L(\ell+1)}.
\end{align}
Define
\[
g^{(\ell+1)}_\mu=\begin{cases} g^{(\ell)}_\mu h^{}_{B_0\cup\ldots\cup B_\ell}\quad&\text{on}\quad B_\mu\quad\text{if}\quad 0\le \mu\le \ell,\\
h_{B_{\ell+1}}^{}&\text{on}\quad B_{\ell+1}\quad\text{if}\quad \mu=\ell+1.\end{cases}
\]Since  \eqref{16.8.16} holds for $k=\ell$, then
\begin{equation*}
\big(g_\mu^{(\ell+1)}\big)^{-1} g_{\mu\nu}^{} g_\nu^{(\ell+1)}=I\quad\text{on}\quad B_\mu\cap B_\nu\quad\text{if}\quad 0\le \mu,\nu\le\ell,
\end{equation*}and from \eqref{16.8.16-} and \eqref{16.8.16++} we get
\begin{equation*}
\big(g_\mu^{(\ell+1)}\big)^{-1} g_{\mu,\ell+1}^{} g_{\ell+1}^{(\ell+1)}=I\quad\text{on}\quad B_\mu\cap B_{\ell+1}\quad\text{for}\quad 0\le \mu\le \ell+1,
\end{equation*}
i.e., we have \eqref{16.8.16} for $k=\ell+1$. \eqref{16.8.16+++} yields \eqref{16.8.16''}  for $k=\ell+1$.
 \end{proof}

\begin{thm}\label{17.8.16--}
 Let $X$ be a one-dimensional Stein space, $Z$ a discrete and closed subset of $X$, and $A:X\to \mathrm{Mat}(n\times n,\C)$ holomorphic. Then  \[H^1\big(X,\Cal C^{\mathrm{Z,\,GCom\,}A}\big)=0.\]
\end{thm}
\begin{proof} Let $f\in Z^1\big(\Cal U,\Cal \Cal C^{Z,\,\mathrm{GCom\,}A}\big)$ be given and let $\Cal U=\{U_i\}_{i\in I}$ be the covering of $f$.  Let $S$ be the set of non-smooth points of $X$.
Since $S$ is discrete and closed in $X$, also $S\cup Z$ is discrete and closed in $X$. Therefore, by Proposition \ref{17.12.15--}, passing to a refinement of $\Cal U$, we may assume that $(S\cup Z)\cap U_i\cap U_j=\emptyset$ for all $i,j\in I$ with $i\not=j$,
 which implies that $f\in Z^1\big(\Cal U,\Cal \Cal C^{S\cup Z,\,\mathrm{GCom\,}A}\big)$.

Now let $\pi:\widetilde X\to X$ be the normalization of $X$, and let
\begin{itemize}
\item[-]$\widetilde A:=A\circ \pi$,
\item[-]$\widetilde U_i:=\pi^{-1}(U_i)$, $i\in I$, and $\widetilde{\Cal U}:=\{\widetilde U_i\}_{i\in I}$,
\item[-]$\widetilde f_{ij}:=f_{ij}\circ \pi$, $i,j\in I$, and $\widetilde f:=\{\widetilde f_{ij}\}_{i,j\in I}$.
\end{itemize}
Since, for each point $\zeta\in X$, $\pi^{-1}(\zeta)$ is finite and $S\cup Z$ is discrete and closed in $X$, then $\pi^{-1}(S\cup Z)$ is discrete and closed in $\widetilde X$, and $\widetilde f\in Z^1\big(\widetilde{\Cal U},\Cal C^{\pi^{-1}(S\cup Z),\,\mathrm{GCom\,}\widetilde A}\big)$. Since $X$ is one-dimensional, $\widetilde X$ is a  Riemann surface, each connected component of which is non-compact. Therefore, by Lemma \ref{4.5.16},
$\widetilde f$ is $\Cal C^{\pi^{-1}(S\cup Z),\,\mathrm{GCom\,}\widetilde A}$-trivial, i.e., there exist $\widetilde g_i\in \Cal C^{\pi^{-1}(S\cup Z),\,\mathrm{GCom\,}\widetilde A}(\widetilde U_i)$, $i\in I$, with
\begin{equation*}\widetilde f_{ij}^{}= \widetilde g_i^{}\widetilde g_j^{-1}\quad\text{on}\quad \widetilde U_i\cap \widetilde U_j,\quad i,j\in I.
\end{equation*} Since
 $\widetilde g_i=I$ in a neighborhood of $\widetilde U_i\cap\pi^{-1}(S\cup Z)$ and since $\pi$ is biholomorphic from $\widetilde X\setminus \pi^{-1}(S)$ onto $X\setminus S$, it follows that there are uniquely defined maps $g^{}_i\in\Cal C^{S\cup Z,\,\mathrm{GCom\,} A}$ such that
\begin{align*}&\widetilde g^{}_i= g^{}_i\circ\pi\quad\text{on}\quad \widetilde U_i\quad\text{for all}\quad i\in I, \quad\text{and}\\
&f_{ij}^{}= g_i^{} g_j^{-1}\quad\text{on}\quad U_i\cap  U_j,\quad i,j\in I.
\end{align*} As $\Cal C^{S\cup Z,\,\mathrm{GCom\,} A}$ is a subsheaf of $\Cal C^{Z,\,\mathrm{GCom\,} A}$, this completes the proof.
\end{proof}

\section{Proof of Theorem \ref{4.10.16} and (hence) of Theorem \ref{28.6.16*} {}}

 Let $X$ and $A$ be as in Theorem \ref{4.10.16},  and let an open covering   $\Cal U=\{U_i\}_{i\in I}$ of $X$ and a $\big(\Cal U,\widehat{\Cal C}^{\mathrm{GCom\,} A}\big)$-cocycle $h=\{h_{ij}\}_{i,j\in I}$ be given.

Let $S$ be the set of non-smooth points of $X$. Since $S$ is discrete and closed in $X$, passing to a refinement of $\Cal U$ (Proposition \ref{17.12.15--}), we may assume that
\begin{equation}\label{18.8.16n}
S\cap U_i^{}\cap U_j^{}=\emptyset\quad\text{for all}\quad i,j\in I\quad\text{with}\quad i\not=j.
\end{equation} Then $h$ can be interpreted as a $\Cal C^{S,\,{\mathrm{GCom\,}A}}$-cocycle (Def. \ref{20.7.16}). By Theorem \ref{17.8.16--} this cocycle is
$\Cal C^{S,\,\mathrm{GCom\,}A}$-trivial.

Now we observe that
$\Cal C^{S,\,\mathrm{GCom\,}A}$ is a subsheaf of $\widehat{\Cal C}^{\mathrm{GCom\,}A}$. Indeed, let an open $U\subseteq X$, $f\in\Cal C^{S,\,\mathrm{GCom\,} A}(U)$ and $\xi\in U$ be given. If $\xi\in S$, then \eqref{19.8.16} is trivial (as then $f\equiv I$ in a neighborhood of $\xi$), and if $\xi$ is a smooth point of $X$, then \eqref{19.8.16} follows from the Smith criterion (Theorem \ref{28.6.16**}, condition (ii)).

So, $h$ is $\widehat{\Cal C}^{\mathrm{GCom\,}A}$-trivial. This  completes the proof of Theorem \ref{4.10.16}, and, as
explained in Section \ref{5.10.16}, also the proof of Theorem \ref{28.6.16*}.

\section{Local counterexamples}\label{local examples}

Let $z$ and $w$ be the canonical complex coordinate functions on $\C^2$.

We begin with the following observation of O. Forster and K. J. Ramspott  \cite[page 159]{FR2}): If $\alpha$, $\beta$ and $\gamma$ are holomorphic functions in a neighborhood of the origin in $\C^2$, which solve the equation
\[
\alpha z^3+\beta w^3+\gamma z^2w^2=0
\]in this neighborhood, then, comparing the coefficients in the Taylor series, it follows easily that $\alpha(0)=\beta(0)=\gamma(0)=0$.
With continuous functions however, this equation can be solved  with $\gamma(0)\not=0$. For example,
\[
\frac{\overline z w^{2}}{\vert z\vert^{2}+\vert w\vert^2}z^3+\frac{\overline w z^{2}}{\vert z\vert^2+\vert w\vert^2}w^3=z^2w^2.
\] We use a $\Cal C^\ell$-version of this.
\begin{leer}\label{2.9.16'''} Let
$\B^2$ the open unit ball in $\C^2$, $\ell\in \N$,
\begin{align*}&A=\begin{pmatrix} z^{2+\ell}w^{2+\ell}&z^{3+\ell}\\w^{3+\ell}&0\end{pmatrix},\qquad B=\begin{pmatrix}, 0&z^{3+\ell}\\w^{3+\ell}&z^{2+\ell}w^{2+\ell}\end{pmatrix},\\
&c^{}_z=\frac{\overline z w^{2+\ell}}{\vert z\vert^{2}+\vert w\vert^2},\quad
c_w^{}=\frac{\overline w z^{2+\ell}}{\vert z\vert^2+\vert w\vert^2},\quad S=\begin{pmatrix} 1&c_w\\-c_z&1\end{pmatrix}.
\end{align*}
Then it is again easy to see  that
\begin{equation}\label{5.1.16-}
c^{}_zz^{3+\ell}+c^{}_ww^{3+\ell}=z^{2+\ell}w^{2+\ell}\quad\text{on}\quad\B^2,
\end{equation}and, comparing the coefficients of the Taylor series\footnote{Below we explain this in detail in the case of Lemmas \ref{6.1.16} and \ref{6.1.16'}, each of which  is stronger than Lemma \ref{5.1.16}.}, we get
\end{leer}
\begin{lem}\label{5.1.16}
Suppose $\alpha$, $\beta$, $\gamma$ are holomorphic functions in a neighborhood  of the origin in $\C^2$ such that
\begin{equation*}
\alpha z^{\ell+3}+\beta w^{\ell+3}+\gamma z^{\ell+2}w^{\ell+2}=0
\end{equation*} in this neighborhood. Then $\alpha(0)=\beta(0)=\gamma(0)=0$.
\end{lem}
Also it is easy to see that the functions
 $c_z$ and $c_w$ are of class $\Cal C^\ell$ on $\C^2$ and that $\vert c_zc_w\vert<1$ on $\B^2$. Hence
$S$ is of class  $\Cal C^\ell$  on $\C^2$, and $S(\zeta)\in \mathrm{GL}(n,\C)$ for all $\zeta\in\B^2$. Moreover,
\begin{align*}
&AS=\begin{pmatrix} z^{2+\ell}w^{2+\ell}-c_zz^{3+\ell}&c_wz^{\ell+2}w^{\ell+2}+z^{3+\ell}\\w^{3+\ell}&c_w^{}w^{3+\ell}\end{pmatrix},\\
&SB=\begin{pmatrix} c_w^{}w^{3+\ell}&z^{3+\ell}+c_wz^{2+\ell}w^{2+\ell}\\w^{3+\ell}&-c_zz^{3+\ell}+z^{2+\ell}w^{2+\ell}\end{pmatrix},
\end{align*} which implies by \eqref{5.1.16-} that
\begin{equation}\label{2.9.16''}SBS^{-1}=A\quad\text{on}\quad \B^2.
\end{equation}Hence $A$ and $B$ are globally $\Cal C^\ell$ similar on $\B^2$. On the other hand, we have
\begin{lem}\label{5.1.16*} Let $U$ be an open neighborhood of the origin in $\C^2$, and $H:U\to \mathrm{Mat}(2\times 2,\C)$ holomorphic. Then:

{\em (i)} If, on $U$, $AH=HB$  or $HA=BH$, then $H(0)=0$.

{\em (ii)} If, on $U$, $AH=HA$ or $AB=BA$, then  $H(0)=\lambda I_2$ for some $\lambda\in \C$.
\end{lem}
\begin{proof} Let $H=\big(\begin{smallmatrix} a&b\\c&d\end{smallmatrix}\big)$. Then
\begin{align}
&\label{3.9.16}AH=\begin{pmatrix} a z^{2+\ell}w^{2+\ell}+c z^{3+\ell}&b z^{2+\ell} w^{2+\ell}+d z^{3+\ell}\\a w^{3+\ell}&b w^{3+\ell}\end{pmatrix},\\
&\label{3.9.16'}HB=\begin{pmatrix} b w^{3+\ell}&a z^{3+\ell}+b z^{2+\ell}w^{2+\ell}\\d w^{3+\ell}&c z^{3+\ell}+dz^{2+\ell} w^{2+\ell}\end{pmatrix},\\
&\label{3.9.16''}HA=\begin{pmatrix} a z^{2+\ell}w^{2+\ell}+b w^{3+\ell}&a z^{3+\ell}\\c z^{2+\ell}w^{2+\ell}+dw^{3+\ell}&cz^{3+\ell}\end{pmatrix},\\
&\label{3.9.16'''}BH=\begin{pmatrix} c z^{3+\ell}&d z^{3+\ell}\\aw^{3+\ell}+cz^{2+\ell}w^{2+\ell}&bw^{3+\ell}+dz^{2+\ell}w^{2+\ell}\end{pmatrix}.
\end{align}  This in particular shows:
\begin{align*}
&\text{if }AH=HB,\text{ then }az^{2+\ell}w^{2+\ell}+c z^{3+\ell}= b w^{3+\ell}=c z^{3+\ell}+d z^{2+\ell} w^{2+\ell},\\
& \text{if }HA=BH,\text{ then }a z^{2+\ell}w^{2+\ell}+b w^{3+\ell}=cz^{3+\ell}=bw^{3+\ell}+dz^{2+\ell}w^{2+\ell},\\
&\text{if }AH=HA, \text{ then }cz^{3+\ell}=bw^{3+\ell} \text{ and }(a-d)z^{3+\ell}=bz^{2+\ell}w^{2+\ell},\\
&\text{if }BH=HB, \text{ then }bw^{3+\ell}=cz^{3+\ell} \text{ and }(d-a)w^{3+\ell}=cz^{2+\ell}w^{2+\ell}.
\end{align*} By Lemma \ref{5.1.16}, this implies
\begin{align*}
&\text{if }AH=HB\text{ or }HA=BH,\text{ then } a(0)=b(0)=c(0)=d(0)=0,\\
&\text{if }AH=HA \text{ or }BH=HB, \text{ then }b(0)=c(0)=0\text{ and }a(0)=d(0).
\end{align*}
\end{proof}

Lemma  \ref{5.1.16*} (i) in particular says that $A$ and $B$ are not locally holomorphically similar at $0$. At the end of this section we prove the following stronger
\begin{thm}\label{6.1.16''} Suppose {\em (a)} $X=\{z^p=w^q\}$, where $p,q\in\N$ such that $\ell+2< q<p$ and $p$, $q$ are  relatively prime,
or {\em (b)} $X$ is the union of  $2\ell+5$ pairwise different one-dimensional linear
subspaces of $\C^2$.

Then the restrictions $A\big\vert_X$ and $B\big\vert_X$ are not locally holomorphically similar at $0$.
\end{thm}

\begin{lem}\label{6.1.16} Let $X=\{z^p=w^q\}$, where $p,q\in\N$ such that $\ell+2< q<p$ and $p$, $q$ are  relatively prime.
Suppose $U$ is a neighborhood of the origin in $\C^2$, and  $\alpha,\beta,\gamma:U\to \C$ are holomorphic such that
\begin{equation}\label{29.11.15'neu}
\alpha z^{\ell+3}+\beta w^{\ell+3}+\gamma z^{\ell+2}w^{\ell+2}=0\quad\text{on}\quad X\cap U.
\end{equation} Then $\alpha(0)=\beta(0)=\gamma(0)=0$.
\end{lem}

\begin{proof} Choose $0<\varepsilon<1$ so small that the closed bidisk $\max(\vert z\vert,\vert w\vert)\le \varepsilon$ is contained in $U$, and let
\[
\sum_{j,k=0}^\infty \alpha_{jk}z^j w^k,\quad \sum_{j,k=0}^\infty \beta_{jk}z^j w^k,\quad\sum_{j,k=0}^\infty \gamma_{jk}z^j w^k
\] be the Taylor series  of $\alpha$, $\beta$ and $\gamma$, respectively.  Then, by \eqref{29.11.15'neu},
\[
\sum_{j,k=0}^\infty \alpha_{jk}z^{j+\ell+3} w^k+\sum_{j,k=0}^\infty \beta_{jk}z^j w^{k+\ell+3}+\sum_{j,k=0}^\infty \gamma_{jk}z^{j+\ell+2} w^{k+\ell+2}=0
\]if $z^p=w^q$ and $\max(\vert z\vert,\vert w\vert)< \varepsilon$. With $z=t^q$ and $w=t^p$ for $0\le t< \varepsilon$, this yields
\begin{equation*}
\sum_{j,k=0}^\infty \alpha_{jk}t_{}^{(j+\ell+3)q+kp} + \sum_{j,k=0}^\infty \beta_{jk}t_{}^{jq+(k+\ell+3)p}+\sum_{j,k=0}^\infty \gamma_{jk}t_{}^{(j+\ell+2)q+(k+\ell+2)p}=0
\end{equation*}for all  $0\le t< \varepsilon$. Comparing the coefficients of $t^{(\ell+3)q}$, $t^{(\ell+3)p}$ and $t^{(\ell+2)(p+q)}$, we get
\begin{align*}
\alpha^{}_{00}+\sum_{(j,k)\in A_\beta} \beta_{jk}+\sum_{(j,k)\in A_\gamma}\gamma_{jk}&=0,\\
\sum_{(j,k)\in B_\alpha} \alpha_{jk}+\beta^{}_{00}+\sum_{(j,k)\in B_\gamma}\gamma^{}_{jk}&=0,\\
\sum_{(j,k)\in C_\alpha} \alpha_{jk}+\sum_{(j,k)\in C_\beta}\beta^{}_{jk}+\gamma^{}_{00}&=0,
\end{align*}
where $A_\beta,\ldots, C_\beta$  are the subsets of $\N\times \N$ defined by
\begin{align*}
(j,k)\in A_\beta&\overset{\mathrm{def}}{\Longleftrightarrow}jq+(k+\ell+3)p=(\ell+3)q\Longleftrightarrow(k+\ell+3)p=(\ell+3-j)q,\\
(j,k)\in A_\gamma&\overset{\mathrm{def}}{\Longleftrightarrow}(j+\ell+2)q+(k+\ell+2)p=(\ell+3)q\Longleftrightarrow (k+\ell+2)p=(1-j)q,\\
(j,k)\in B_\alpha&\overset{\mathrm{def}}{\Longleftrightarrow}(j+\ell+3)q+kp=(\ell+3)p\Longleftrightarrow (j+\ell+3)q=(\ell+3-k)p,\\
 (j,k)\in B_\gamma&\overset{\mathrm{def}}{\Longleftrightarrow}(j+\ell+2)q+(k+\ell+2)p =(\ell+3)p\Longleftrightarrow (j+\ell+2)q=(1-k)p,\\
(j,k)\in C_\alpha&\overset{\mathrm{def}}{\Longleftrightarrow}(j+\ell+3)q+kp=(\ell+2)(p+q)\Longleftrightarrow (j+1)q=(\ell+2-k)p,\\
(j,k)\in C_\beta&\overset{\mathrm{def}}{\Longleftrightarrow}jq+(k+\ell+3)p=(\ell+2)(p+q)\Longleftrightarrow (\ell+2-j)q=(k+1)p.
\end{align*}
It is sufficient to prove that $A_\beta=A_\gamma=B_\alpha=B_\gamma=C_\alpha=C_\beta=\emptyset$.

Assume $(k+\ell+3)p=(\ell+3-j)q$. Contrary to  $q<p$, then it follows
\[
p=\frac{\ell+3-j}{k+\ell+3}q\le \frac{\ell+3}{k+\ell+3}\le q.
\]

Assume  $(k+\ell+2)p=(1-j)q$. Contrary to $p>p/2$, then it follows
\[
p=\frac{1-j}{k+\ell+2}q\le \frac{q}{2}<\frac{p}{2}.
\]

Assume $(j+\ell+3)q=(\ell+3-k)p$. Since $p$ and $q$ are relatively prime, this implies that $j+\ell+3=np$, for some integer $n\in \N^*$. $n=1$ is not possible, for this would imply that $p=j+\ell+3\le \ell+3\le q<p$.  $n\ge 2$ is also impossible, as this would imply that
$p\ge \ell+3\ge j+\ell+3\ge 2p$.

Assume $(j+\ell+2)q=(1-k)p$. This implies that $k=0$ and therefore $(j+\ell+2)q=p$, which is not possible, since $p$ and $q$ are relatively prime.

Assume $(j+1)q=(\ell+2-k)p$.  As $p$ and $q$ are relatively prime, this implies that $\ell +2-k$ is positive and can be divided by $q$. In particular, $\ell+2-k\ge q$, which is not possible, for $\ell+2-k<q$.

Assume $(\ell+2-j)q=(k+1)p$. Since $p$ and $q$ are relatively prime, this implies that $\ell+2-j=np$ for some $n\in\N^*$ and, further,  $p>\ell+2\ge\ell+2-j=np\ge p$, contrary to $q<p$.
\end{proof}

\begin{lem}\label{6.1.16'}Let $t_1,\ldots,t_{2\ell+5}$ be pairwise different complex numbers, and
\[
X:=\bigcup_{j=1}^{2\ell+5}\{w=t_j z\}.
\]
Suppose $U$ is a neighborhood of the origin in $\C^2$, and  $\alpha,\beta,\gamma:U\to \C$ are holomorphic such that
\begin{equation}\label{29.11.15'}
\alpha z^{\ell+3}+\beta w^{\ell+3}+\gamma z^{\ell+2}w^{\ell+2}=0\quad\text{on}\quad X\cap U.
\end{equation} Then $\alpha(0)=\beta(0)=\gamma(0)=0$.
\end{lem}

\begin{proof}

To prove that $\alpha(0)=0$, we assume that $\alpha(0)\not=0$. Setting $b=\beta/\alpha$ and $c=\gamma/\alpha$, then we get holomorphic functions $b, c$ in a neighorhood $V\subseteq U$ of $0$ such that
\[
z^{\ell+3}=c z^{2+\ell}w^{\ell+2}-b w^{\ell+3}=0\quad \text{on}\quad X\cap V.
\]It follows that, for $1\le j\le 2\ell+ 5$ and all $\zeta$ in some neighborhood  of zero in the complex plane,
\[\zeta^{\ell+3}=c(\zeta,t_j\zeta) \zeta^{2\ell+4}t_j^{\ell+2}-b(\zeta,t_j\zeta)\zeta^{\ell+3}t_j^{\ell+3}
\]and, hence,
\[1=c(\zeta,t_j\zeta) \zeta^{\ell+1}t_j^{\ell+2}-b(\zeta,t_j\zeta)t_j^{\ell+3}.
\]Hence, $1=-t_j^{\ell+3}\beta(0,0)$ for $1\le j\le 2\ell+5$. This implies that $\beta(0,0)\not=0$ and $t_1,\ldots,t_{2\ell+5}$ are
solutions of the equation
\[
t^{\ell+3}=-\frac{1}{\beta(0,0)}.
\]As $2\ell+5>\ell+3$ and the numbers $t_j$ are pairwise different, this is impossible.

Changing the roles of $z$ and $w$, one proves in the same way  that $\beta(0)=0$.

Finally we assume that $\gamma(0)\not=0$. Setting $a=\alpha/\gamma$ and $b=\beta/\gamma$, then we
get holomorphic functions $a,b$ in a neighborhood $V\subseteq U$ of $0$ such that
\[
a z^{\ell+3}+b w^{\ell+3}=z^{\ell+2}w^{\ell+2}\quad\text{on}\quad X\cap V.
\]It follows that, for $1\le j\le 2\ell+ 5$ and all $\zeta$ in some neighborhood $\Omega$ of zero in the complex plane
\[
a(\zeta,t_j\zeta ) \zeta^{\ell+3}+b(\zeta,t_j\zeta) \zeta^{\ell+3}t_j^{\ell+3}=\zeta^{2\ell+4}t_j^{\ell+2}
\]and, hence,
\begin{equation*}\label{8.11.15}
a(\zeta,t_j\zeta) +b(\zeta,t_j\zeta)t_j^{\ell+3} =\zeta^{\ell+1}t_j^{\ell+2}.
\end{equation*} If $\sum a_{\mu\nu}z^\mu w^\nu$ and $\sum b_{\mu\nu}z^\mu w^\nu$ are the Taylor series at the origin
of $a$ and $b$, respectively,
this means that
\[\sum_{\mu,\nu=0}^\infty \Big(a_{\mu\nu}t_j^\nu+b_{\mu\nu} t_j^{\nu+\ell+3}\Big)\zeta^{\mu+\nu}
=\zeta^{\ell+1}t_j^{\ell+2}
\]for all  $1\le j\le 2\ell+5$ and $\zeta\in\Omega$. Comparing the coefficients of $\zeta^{\ell+1}$, this yields
\[
\sum_{\mu+\nu=\ell+1}a_{\mu\nu}t_j^\nu+\sum_{\mu+\nu=\ell+1}b_{\mu\nu}t_j^{\nu+\ell+3}=t_j^{\ell+2},\quad\quad 1\le j\le 2\ell+5.
\]i.e.,
\[
\sum_{\nu=0}^{\ell+1}\alpha^{}_{\ell+1-\nu,\nu}t_j^\nu+\sum_{\nu=\ell+3}^{2\ell+4}\beta^{}_{2\ell+4-\nu,\nu-\ell-3}t_j^{\nu}=t_j^{\ell+2},
\quad\quad 1\le j\le 2\ell+5.
\]
Hence, the system of $2\ell+5$ linear equations in $2\ell+5$ variables
\[
\sum_{\nu=0}^{2\ell+4}t_j^\nu x_\nu=0, \quad\quad 1\le j\le 2\ell+5
\]has a non-trivial solution, namely one with $x_{\ell+2}^{}=-1$. This not possible, as
\[
\det\begin{pmatrix}1&t^{}_1&t_1^2&\ldots &t^{2\ell+4}_1\\&&\ldots&&\\1&t_{2\ell+5}^{}&t_{2\ell+5}^2&\ldots& t_{2\ell+5}^{2\ell+4}\end{pmatrix}
=\prod_{1\le i<j\le 2\ell+5}(t_i-t_j)\not=0.\]
\end{proof}

{\em Proof of Theorem \ref{6.1.16''}.} Assume there exist a neighborhood $U$ of the origin in $\C^2$ and a holomorphic map $H:U\to\mathrm{GL}(2,\C)$ such that $H^{-1}AH=B$ on $X\cap U$. If $H=\big(\begin{smallmatrix} a&b\\c&d\end{smallmatrix}\big)$, then, by \eqref{3.9.16} and \eqref{3.9.16'}, in particular, it follows that
\[
az^{2+\ell}w^{2+\ell}+c z^{3+\ell}= b w^{3+\ell}=c z^{3+\ell}+d z^{2+\ell} w^{2+\ell}\quad\text{on}\quad X\cap U,
\] which implies by Lemmas \ref{6.1.16} and \ref{6.1.16'} that $a(0)=b(0)=c(0)=d(0)=0$, i.e., $H(0)=0$, which contradicts the assumption that $H(0)$ is invertible.
\qed

\section{A global counterexample}\label{global example}

Let $v_1,v_2,v_3$ denote the canonical complex coordinate functions on $\C^3$,  $x_j:=\rea v_j$ and $y_j:=\im v_j$. Set
\begin{align*}h=v_1+&iv_2,\quad h^*=v_1-iv_2,\\
\s^2=\big\{y_1=y_2=y_3=0\big\}&\cap \big\{x_1^2+x_2^2+x_3^2=1\big\},\quad \s^1=\s^2\cap \big\{x_3=0\big\}.
\end{align*}
Then $hh^*= x_1^2+x_2^2=1$ on $\s^1$. Therefore we can find a neighborhood $N(\s^2)$ in $\C^3$ of $\s^2$ and  $\varepsilon>0$   such that
\begin{equation}\label{7.1.16}
\big\vert  hh^*-1\big\vert<\frac{1}{2}\quad\text{on}\quad N(\s^2)\cap \{-2\varepsilon<x_3<2\varepsilon\}.
\end{equation}
Set
\[
\rho=\big(x_1^2+x_2^2+x_3^2-1\big)^3+y_1^2+y_2^2+y_3^2.
\]Then $\s^2=\{\rho=0\}$ and, making $\varepsilon$ smaller, we can achieve that
\[
\s^2_\varepsilon=\{\rho<\varepsilon\}\subseteq N(\s^2).
\]Moreover, we can choose $\varepsilon$ so small that $\rho$ is strictly plurisubharmonic in $\s^2_\varepsilon$. Then
$\s^2_\varepsilon$ is Stein. Set \[U_+=\s^2_\varepsilon\cap\{x_3>-\varepsilon\}\quad\text{and}\quad U_-=\s^2_\varepsilon\cap \{x_3<\varepsilon\}.
\]

\begin{lem}\label{8.1.16} {\em(i)} There exist holomorphic  $a_\pm,b_\pm,c_\pm,d_\pm:U_\pm\to \C$ such that
\begin{align}&\label{8.1.16'}\begin{pmatrix}a_\pm(\zeta)&b_\pm(\zeta)\\c_\pm(\zeta)&d_\pm(\zeta)\end{pmatrix}\in \mathrm{GL}(2,\C)\quad\text{for all}\quad \zeta\in U_\pm,\\&\label{8.1.16''} \begin{pmatrix}h&0\\0&h^*\end{pmatrix}
=\begin{pmatrix}a_+&b_+\\c_+&d_+\end{pmatrix}
\begin{pmatrix}a_-&b_-\\c_-&d_-\end{pmatrix}^{-1}\quad\text{on}\quad U_+\cap U_-.
\end{align}
{\em (ii)} There do not exist continuous functions $f_\pm^{}:U_\pm^{}\to \C^*$ such that
\begin{equation}\label{9.1.16}h=\frac{f_+}{f_-}\quad\text{on}\quad U_+\cap U_-.
\end{equation}
\end{lem}

\begin{proof} (i) Since $\s_\varepsilon^2$ is Stein and $\s^2_\varepsilon=U_+\cup U_-$, by Grauert's Oka principle [Satz I]\cite{Gr}, [Theorem 5.3.1 (ii)]\cite{Fc}, it is sufficient to  find a continuous  $C_+:U_+\to \mathrm{GL}(2,\C)$ with
\begin{equation}\label{8.1.16-}
\begin{pmatrix}h&0\\0&h^*\end{pmatrix}=C^{}_+\quad\text{on}\quad U_+\cap U_-,
\end{equation}which can be done as follows: Take a continuous function $\chi:\R\to[0,1]$ such that $\chi(t)=1$ if $t\le \varepsilon$  and $\chi(t)=0$ if $ t\ge 2\varepsilon$, and define
\[
C_+(\zeta)=\begin{pmatrix}\chi\big(x_3(\zeta)\big)h(\zeta)&1-\chi\big(x_3(\zeta)\big)\\\chi\big(x_3(\zeta)\big)-1&\chi\big(x_3(\zeta)\big)h^*(\zeta)
\end{pmatrix}\quad \text{for}\quad\zeta\in U_+.
\]If $\zeta\in U_+\cap U_-$, then $-\varepsilon<x_3(\zeta)<\varepsilon$ and therefore $\chi\big(x_3(\zeta)\big)=1$, which implies \eqref{8.1.16-}.
It remains to prove that $\det C_+(\zeta)\not=0$  for all $\zeta\in U_+$.
If $\zeta\in U_+$ with $x_3(\zeta)<2\varepsilon$, then,  by \eqref{7.1.16}, $\rea\big(h(\zeta)h^*(\zeta)\big)>1/2$, which yields
\[
\rea\det C_+(\zeta)\ge \frac{1}{2}\Big(\chi\big(x_3(\zeta)\big)\Big)^2+\Big(1-\Big(\chi\big(x_3(\zeta)\big)\Big)^2\ge \frac{1}{2}.
\]If  $\zeta\in U_+$ with $x_3(\zeta)\ge 2\varepsilon$, then $\chi\big(x_3(\zeta)\big)=0$ and, hence, $\det C_+(\zeta)=1$.

 (ii)
Assume such functions exist. Then, for $0\le s\le 1$,
 we have  continuous closed curves $\gamma^+_s:[0,2\pi]\to \C^*$ and $\gamma^-_s:[0,2\pi]\to \C^*$, well-defined by
 \begin{align*}
 &\gamma^+_s(t)=f_+^{}\Big(\big(\sqrt{1-s^2}\cos t,\sqrt{1-s^2}\sin t,s\big)\Big),\\
 &\gamma^-_s(t)=f_-^{}\Big(\big(\sqrt{1-s^2}\cos t,\sqrt{1-s^2}\sin t,-s\big)\Big).
 \end{align*}
 Let
 \[
 \mathrm{Ind\,}\gamma_s^\pm:=\frac{1}{2\pi}\int_0^{2\pi}\frac{(\gamma_s^{\pm})'(t)}{\gamma_s^\pm(t)}\,dt
 \] be the winding number of $\gamma_s^\pm$. It is clear that $\mathrm{Ind\,}\gamma_s^\pm$ depends continuously on $s$, and
 it is well known that $\mathrm{Ind\,}\gamma_s^\pm$ is always an integer. Therefore
 \[
 \mathrm{Ind\,}\gamma_{1}^+=\mathrm{Ind\,}\gamma_0^+\quad\text{and}\quad \mathrm{Ind\,}\gamma_{1}^-=\mathrm{Ind\,}(\gamma_0^-).
 \] Since $\gamma^+_1$ and $\gamma_1^-$ are constant, it follows that $\mathrm{Ind\,}(\gamma_0^+)=\mathrm{Ind\,}(\gamma_0^-)=0$ and, hence,
 \begin{equation}\label{9.1.16'}
 \mathrm{Ind\,}\frac{\gamma_0^+}{\gamma_0^-}=0.
 \end{equation}By definition of $h$, $h\big(\cos t,\sin t,0\big)=\cos t+i\sin t=e^{it}$.
 By \eqref{9.1.16} this yields
\[
e^{it}=\frac{f_+\big(\cos t,\sin t,0\big)}{f_-\big(\cos t,\sin t,0\big)}=\frac{\gamma^+_0(t)}{\gamma^-_0(t)},\qquad 0\le t\le 2\pi,
\]which contradicts \eqref{9.1.16'}.
\end{proof}
Now, using also the notations introduced in  Section \ref{2.9.16'''}, we set
\[X=\s^2_\varepsilon\times \B^2,\qquad  X_\pm=U_\pm\times \B^2,\] and, for $(\zeta,\eta)\in X$,
\[\widetilde A(\zeta,\eta)=A(\eta),\quad \widetilde B(\zeta,\eta)=B(\eta),\quad
\widetilde h(\zeta,\eta)=h(\zeta),\quad\widetilde h^*(\zeta,\eta)=h^*(\zeta).\]
Further, let $a_\pm,b_\pm,c_\pm,d_\pm$ be as in Lemma \ref{8.1.16} (i), and define holomorphic maps $\Theta_\pm:X_\pm\to \mathrm{Mat}(4\times 4,\C)$ by the block matrices
\[
\Theta_\pm(\zeta,\eta)=\begin{pmatrix}a_\pm(\zeta)I_2& b_\pm(\zeta)I_2\\ c_\pm(\zeta)I_2&d_\pm(\zeta)I_2\end{pmatrix},\quad (\zeta,\eta)\in X_\pm.
\]Then, by \eqref{8.1.16'} and \eqref{8.1.16''},
$\Theta_\pm(\zeta,\eta)\in \mathrm{GL}(4,\C)$ for all $(\zeta,\eta)\in X_\pm$, and
\begin{equation}\label{10.1.16}
\begin{pmatrix}\widetilde h I_2&0\\0&\widetilde h^*I_2\end{pmatrix}=\Theta^{}_+\Theta_-^{-1}\quad \text{on}\quad X_+\cap X_-.
\end{equation}Since, obviously,
\[
\begin{pmatrix}\widetilde A&0\\0&\widetilde B\end{pmatrix}\begin{pmatrix}\widetilde h I_2&0\\0&\widetilde h^*I_2\end{pmatrix}=\begin{pmatrix}\widetilde h I_2&0\\0&\widetilde h^*I_2\end{pmatrix}\begin{pmatrix}\widetilde A&0\\0&\widetilde B\end{pmatrix}\quad\text{on}\quad X,
\]this implies that
\begin{equation}\label{2.9.16}
\Theta_+^{}\begin{pmatrix}\widetilde A&0\\0&\widetilde B\end{pmatrix}\Theta_+^{-1}=\Theta_-^{}\begin{pmatrix}\widetilde A&0\\0&\widetilde B\end{pmatrix}\Theta_-^{-1}\quad\text{on}\quad X_+\cap X_-.
\end{equation}Let $\Phi:X\to \mathrm{Mat}(4\times 4,\C)$ be defined by the two sides of \eqref{2.9.16}.

\begin{thm}\label{10.1.16'} $\Phi$ and $\big(\begin{smallmatrix}\widetilde A&0\\0&\widetilde B\end{smallmatrix}\big)$ are

{\em (a)} locally holomorphically similar on $X$,

{\em (b)} globally $\Cal C^\ell$ similar on $X$,

{\em(c)} not globally $\Cal C^\infty$ similar on $X$.
\end{thm}
\begin{proof}The local holomorphic similarity is clear by definition of $\Phi$.

To prove (b), let $S$ be as in Section \ref{2.9.16'''} and $\widetilde S(\zeta,\eta):=S(\eta)$ for $(\zeta,\eta)\in X$.
Since  $a_\pm(\zeta)I_2$, $b_\pm(\zeta)I_2$, $c_\pm(\zeta)I_2$ and $d_\pm(\zeta)I_2$ commute with $A(\eta)$, we have
\begin{equation}\label{10.1.16+}
\begin{pmatrix}\widetilde A&0\\0&\widetilde A\end{pmatrix}\Theta_\pm=\Theta_\pm\begin{pmatrix}\widetilde A&0\\0&\widetilde A\end{pmatrix}\quad\text{on}\quad U_\pm.
\end{equation}Moreover, it is clear that $\widetilde S\widetilde h^*I_2= \widetilde h^*I_2 \widetilde S$ and therefore
\[
\begin{pmatrix}I_2&0\\0&\widetilde S\end{pmatrix}\begin{pmatrix}\widetilde hI_2&0\\0&\widetilde h^*I_2\end{pmatrix}=\begin{pmatrix}\widetilde hI_2&0\\0&\widetilde h^*I_2\end{pmatrix}\begin{pmatrix}I_2&0\\0&\widetilde S\end{pmatrix}\quad\text{on}\quad X,
\]which implies by \eqref{10.1.16} that
\[
\Theta_+^{-1}\begin{pmatrix}I_2&0\\0&\widetilde S\end{pmatrix}\Theta_+^{}=\Theta_-^{-1}\begin{pmatrix}I_2&0\\0&\widetilde S\end{pmatrix}\Theta_-^{}
\quad\text{on}\quad X_+\cap X_-
\]and further
\[
\begin{pmatrix}I_2&0\\0&\widetilde S^{-1}\end{pmatrix}\Theta_+^{-1}\begin{pmatrix}I_2&0\\0&\widetilde S\end{pmatrix}\Theta_+^{}=
\begin{pmatrix}I_2&0\\0&\widetilde S^{-1}\end{pmatrix}\Theta_-^{-1}\begin{pmatrix}I_2&0\\0&\widetilde S\end{pmatrix}\Theta_-^{}
\quad\text{on}\quad X_+\cap X_-.
\]
Let $\Psi:X\to \mathrm{GL}(4,\C)$ be the $\Cal C^\ell$ map defined by the two sides of the last equality. Then, by  \eqref{2.9.16''},
\begin{equation*}
\Psi^{-1}\begin{pmatrix}\widetilde A &0\\0&\widetilde B\end{pmatrix}\Psi=\Theta_\pm^{-1}\begin{pmatrix}I_2&0\\0&\widetilde S^{-1}\end{pmatrix}\Theta_\pm^{}\begin{pmatrix}\widetilde A &0\\0&\widetilde A\end{pmatrix}
\Theta_\pm^{-1}\begin{pmatrix}I_2&0\\0&\widetilde S\end{pmatrix}\Theta_\pm^{}\quad\text{on}\quad X_\pm.
\end{equation*}In view of \eqref{10.1.16+}, this implies that
\[\Psi^{-1}\begin{pmatrix}\widetilde A &0\\0&\widetilde B\end{pmatrix}\Psi=\Theta_\pm^{-1}\begin{pmatrix}I_2&0\\0&\widetilde S^{-1}\end{pmatrix}\begin{pmatrix}\widetilde A &0\\0&\widetilde A\end{pmatrix}\begin{pmatrix}I_2&0\\0&\widetilde S\end{pmatrix}\Theta_\pm^{}\quad\text{on}\quad X_\pm,
\] and further, again by \eqref{2.9.16''},
\[\Psi^{-1}\begin{pmatrix}\widetilde A &0\\0&\widetilde B\end{pmatrix}\Psi=\Theta_\pm^{-1}\begin{pmatrix}\widetilde A &0\\0&\widetilde B\end{pmatrix}
\Theta_\pm^{}\quad\text{on}\quad X_\pm.
\]By definition of $\Phi$, this means that $\Psi^{-1}\big(\begin{smallmatrix}\widetilde A&0\\0&\widetilde A\end{smallmatrix}\big)\Psi=\Phi$ on $X$, which completes the proof of (b).

To prove (c), we assume that  $\Phi$ and $\big(\begin{smallmatrix}\widetilde A&0\\0&\widetilde B\end{smallmatrix}\big)$ are globally $\Cal C^\infty$ similar on $X$. Since $X$ is Stein, then, by Theorem \ref{24.8.16+}, $\Phi$ and $\big(\begin{smallmatrix}\widetilde A&0\\0&\widetilde B\end{smallmatrix}\big)$ are even globally holomorphically  similar on $X$, i.e., we have a holomorphic map $\Theta:X\to \mathrm{GL}(4,\C)$ with
\[
\Theta^{-1}\Phi\Theta=\begin{pmatrix}\widetilde A&0\\0&\widetilde B\end{pmatrix}\quad\text{on}\quad X.
\]By definition of $\Phi$ this means that
\[
\Theta^{-1}_{}\Theta_\pm^{}\begin{pmatrix}\widetilde A&0\\0&\widetilde B\end{pmatrix}\Theta_\pm^{-1}\Theta=\begin{pmatrix}\widetilde A&0\\0&\widetilde B\end{pmatrix}\quad\text{on}\quad X_\pm,
\]i.e.,
\begin{equation*}\label{11.1.16}
\Theta^{-1}_{}\Theta_\pm^{}\begin{pmatrix}\widetilde A&0\\0&\widetilde B\end{pmatrix}=\begin{pmatrix}\widetilde A&0\\0&\widetilde B\end{pmatrix}\Theta^{-1}\Theta_\pm^{}\quad\text{on}\quad X_\pm.
\end{equation*}
If $C_\pm$, $D_\pm$, $E_\pm$, $F_\pm$ are the $2\times 2$ matrices with
\[
\Theta^{-1}_{}\Theta_\pm^{}=\begin{pmatrix}C_\pm&D_\pm\\E_\pm&F_\pm\end{pmatrix},
\]then this means that, on $X_\pm$,
\[C_\pm \widetilde A=\widetilde AC_\pm,\quad F_\pm \widetilde B=\widetilde B F_\pm,\quad E_\pm \widetilde A=\widetilde BE_\pm,\quad D_\pm \widetilde B=\widetilde AD_\pm,
\]i.e., for each fixed $\zeta\in U_\pm$, we have, on $\B^2$,
\begin{align*}&C_\pm(\zeta,\cdot)A= AC_\pm(\zeta,\cdot),\quad F_\pm(\zeta,\cdot) B=B F_\pm(\zeta,\cdot),\\
&E_\pm(\zeta,\cdot) A= BE_\pm(\zeta,\cdot),\quad D_\pm(\zeta,\cdot) B= AD_\pm(\zeta,\cdot).
\end{align*}
By Lemma \ref{5.1.16*} this yields that, for each $\zeta\in U_\pm$, there exist $\gamma_\pm(\zeta), \varphi_\pm(\zeta)\in \C$ with
\begin{equation}\label{3.9.16+}
\Theta_{}^{-1}(\zeta,0)^{}\Theta_{\pm}^{}(\zeta,0)=\begin{pmatrix}\gamma_\pm(\zeta)I_2&0\\0&\varphi_\pm(\zeta)I_2\end{pmatrix}\quad\text{for all}\quad \zeta\in U_\pm.
\end{equation} Since the maps $\Theta^{-1}_{}\Theta_\pm^{}$ are holomorphic and have invertible values on $X_\pm$, the so defined functions $\gamma_\pm$ and $\varphi_\pm$ must be holomorphic and different from zero on $U_\pm$. Moreover, by \eqref{10.1.16}, it follows from the equations \eqref{3.9.16+} that, for $\zeta\in U_+\cap U_-$,
\begin{equation*}
\Theta(\zeta,0)^{-1}_{}\begin{pmatrix}h(\zeta)I_2&0\\0&h^*(\zeta)I_2\end{pmatrix}
\Theta(\zeta,0)=\begin{pmatrix}\gamma^{}_+(\zeta,0)\gamma^{}_-(\zeta,0)^{-1}I_2&0\\0&\varphi^{}_+(\zeta)\varphi_-^{}(\zeta,0)^{-1}I_2
\end{pmatrix}.
\end{equation*}In particular, for each $\zeta\in U_+\cap U_+$, the matrices
\[\begin{pmatrix}h(\zeta)I_2&0\\0&h^*(\zeta)I_2\end{pmatrix}\quad\text{and}\quad
\begin{pmatrix}\gamma^{}_+(\zeta,0)\gamma^{}_-(\zeta,0)^{-1}&0\\0&\varphi^{}_+(\zeta)\varphi_-^{}(\zeta,0)^{-1}
\end{pmatrix}
\] are similar; hence they have the same eigenvalues. In particular,
\begin{equation}\label{13.1.16-}
h(\zeta)\in\Big\{\gamma^{}_+(\zeta,0)\gamma^{}_-(\zeta,0)^{-1},\varphi^{}_+(\zeta)\varphi_-^{}(\zeta,0)^{-1}\Big\}\quad\text{if}\quad \zeta\in U_+\cap U_-.
\end{equation} Consider the open sets
\[V\gamma:=\big\{\zeta\in U_+\cap U_-\,\vert\, h(\zeta)=\gamma^{}_+(\zeta,0)\gamma^{}_-(\zeta,0)^{-1}\big\}
\] and
\[V_\varphi:=\big\{\zeta\in U_+\cap U_-\,\vert\, h(\zeta)=\varphi^{}_+(\zeta,0)\varphi^{}_-(\zeta,0)^{-1}\big\}.
\]Then, by \eqref{13.1.16-}, $V_\gamma\cup V_\varphi= U_+\cap U_-$. Hence, at least one of these sets, say $V_\gamma$, is non-empty. Since the functions $h$ and  $\gamma^{}_+(\cdot,0)\gamma^{}_-(\cdot,0)^{-1}$  both are holomorphic on $U_+\cap U_-$ and $U_+\cap U_-$ is connected, it follows that $h(\zeta)=\gamma^{}_+(\zeta,0)\gamma^{}_-(\zeta,0)^{-1}$ for all $\zeta\in U_+\cap U_-$, which is not possible, by Lemma \ref{8.1.16} (ii).
\end{proof}

\end{document}